\documentclass[11pt]{amsart}
\usepackage{amsmath,amssymb,amsthm,enumerate}
\usepackage[utf8]{inputenc}
\usepackage[english]{babel}
\usepackage{xy} 
\usepackage{mathrsfs} 
\usepackage{xspace}
\usepackage{tikz}
\usepackage{adjustbox,booktabs}
\usepackage[pagebackref=true, colorlinks]{hyperref}
\title[CAT(0) cube complexes and inner amenability]{CAT(0) cube complexes and inner amenability}

\author[B. Duchesne]{Bruno Duchesne}
\author[R. Tucker-Drob]{Robin Tucker-Drob}
\author[P. Wesolek]{Phillip Wesolek}
\date{March 2019}

\newtheorem{thm}{Theorem}[section]

\newtheorem{prop}[thm]{Proposition}
\newtheorem{lem}[thm]{Lemma}
\newtheorem{cor}[thm]{Corollary}

\theoremstyle{definition}
\newtheorem{defn}[thm]{Definition}

\newtheorem{rmk}[thm]{Remark}

\newtheorem*{ack}{Acknowledgments}

\newtheorem{exs}[thm]{Example}

\newcommand{\Zb}{\mathbb{Z}}
\newcommand{\Nb}{\mathbb{N}}
\newcommand{\Rb}{\mathbb{R}}
\newcommand{\Qb}{\mathbb{Q}}

\newcommand{\mc}[1]{\mathcal{#1}}

\newcommand{\mf}[1]{\mathfrak{#1}}

\newcommand{\cat}{{\upshape CAT($0$)}\xspace}
\DeclareMathOperator{\Fix}{Fix}

\DeclareMathOperator{\Hyp}{Hyp}
\DeclareMathOperator{\Ell}{Ell}
\DeclareMathOperator{\Par}{Par}
\DeclareMathOperator{\Aut}{Aut}

\DeclareMathOperator{\Isom}{Isom}
\newcommand{\bo}{\partial}

\newcommand{\acts}{\curvearrowright}

\usepackage{mathtools}
\newcommand{\defeq}{\coloneqq}

\usepackage{comment}

\begin{document}

\maketitle
\begin{abstract}We here consider inner amenability from a geometric and group theoretical perspective. We prove that for every non-elementary action of a group $G$ on a finite dimensional irreducible \cat cube complex, there is a nonempty $G$-invariant closed convex subset such that every conjugation invariant mean on $G$ gives full measure to the stabilizer of each point of this subset. Specializing our result to trees leads to a complete characterization of inner amenability for HNN-extensions and amalgamated free products. One novelty of the proof is that it makes use of the existence of certain idempotent conjugation-invariant means on $G$.

We additionally obtain a complete characterization of inner amenability for permutational wreath product groups. One of the main ingredients used for this is a general lemma which we call the location lemma, which allows us to ``locate'' conjugation invariant means on a group $G$ relative to a given normal subgroup $N$ of $G$. We give several further applications of the location lemma beyond the aforementioned characterization of inner amenable wreath products.
\end{abstract}
\tableofcontents
\addtocontents{toc}{\protect\setcounter{tocdepth}{1}}

\section{Introduction}

A discrete group $G$ is said to be \textbf{inner amenable} if there exists an atomless mean on $G$ which is invariant for the action of $G$ on itself by conjugation. This notion was isolated by Effros in \cite{MR0355626}\footnote{Our definition is slightly different from the definition given in \cite{MR0355626}, where the mean is not required to be atomless, but rather supported on $G - \{ 1_G \}$. However, the two definitions coincide for infinite conjugacy class (ICC) groups, which were the main concern of \cite{MR0355626}.} in order to elucidate Murray and von Neumann's proof that the group von Neumann algebra of the free group on two generators has no nontrivial asymptotically central sequences \cite{MvN43}. Similar connections between inner amenability and central sequences were later found by Choda \cite{Ch82} and Jones and Schmidt \cite{JS87} in the context of ergodic theory. These connections to operator algebras and ergodic theory have continued to provide a rich context and motivation for the study of inner amenability; see, e.g., \cite{Va12, Ki12, Ki14a, Ki14b, CSU16, Ki17, Ki19, Tu14, HI16, Oz16, DV18, IS18, BIP18, KeTD19, KiTD18}. Perhaps because of this, inner amenability has been studied primarily by virtue of its relevance to these two fields (with a few exceptions, e.g., \cite{BH86, BedosHarpeErratum, MR2226017, GH91}). In this article, by contrast, we explore inner amenability from the perspectives of geometry and group theory.

%
%
%

\subsection{Conjugation invariant means on groups acting on trees}\label{subsec:introtree}  We give a complete characterization of inner amenability for groups built via amalgamated free products and HNN-extensions.

We say that an amalgamated free product $G =A*_H B$ is {\bf nondegenerate} if $H\neq A$, $H\neq B$, and the index of $H$ in either $A$ or $B$ is at least three.

\begin{thm}[Corollary~\ref{cor:amalgam_hnn}]\label{thm:amalgam} Let $G =A*_H B$ be a nondegenerate amalgamated free product. Then:
\begin{enumerate}
\item  Every conjugation invariant mean on $G$ concentrates on $H$.
\item $G$ is inner amenable if and only if there exist conjugation invariant, atomless means $m_A$ on $A$ and $m_B$ on $B$ with $m_A(H)=m_B(H)=1$, and $m_A(E)=m_B(E)$ for every $E\subseteq H$.
\end{enumerate}
In particular, if $G$ is inner amenable, then so are each of the groups $A$, $B$, and $H$.
\end{thm}

Let $H$ be a subgroup of a group $K$ and let $\phi : H \rightarrow K$ be an injective group homomorphism. The associated HNN-extension
\[
\mathrm{HNN}(K,H,\phi ) \defeq  \langle K , t \, | \, tht^{-1}  = \varphi (h), \ (h\in H ) \rangle
\]
is said to be {\bf ascending} if either $H=K$ or $\phi (H)=K$. Otherwise, it is called {\bf non-ascending}.

\begin{thm}[Corollary~\ref{cor:amalgam_hnn}]\label{thm:HNN}
Let $G =\mathrm{HNN}(K,H,\phi )$ be a non-ascending HNN extension. Then:
\begin{enumerate}
\item Every conjugation invariant mean on $G$ concentrates on $H$.
\item $G$ is inner amenable if and only if there exists a conjugation invariant, atomless mean $m$ on $K$ with $m(H)=1$, and $m(E)=m(\phi (E))$ for every $E\subseteq H$.
\end{enumerate}
In particular, if $G$ is inner amenable, then so are the groups $K$ and $H$.
\end{thm}

Theorems \ref{thm:amalgam} and \ref{thm:HNN} are group theoretical consequences of a more general geometric statement, regarding groups acting on trees. Given a group $G$ and a $G$-set $X$,  we denote by $G _x$ the stabilizer subgroup of $G $ at $x\in X$.

\begin{thm}\label{thm:tree}
Suppose that a group $G$ acts by automorphisms on a tree $T$. Assume that $G$ does not fix a vertex, an edge, an end, or a pair of ends. Then there is a nonempty $G$-invariant subtree $T_0$ of $T$ such that $m(G_x)=1$ for every conjugation invariant mean $m$ on $G$ and every vertex $x$ of $T_0$.
\end{thm}

This theorem follows directly from Theorem~\ref{thm:trees} by considering the unique minimal $G$-invariant subtree $T_0$ of $T$.

\subsection{Groups acting on \cat cube complexes}
Theorem \ref{thm:tree} is itself a special case of the following general theorem concerning groups acting on finite dimensional \cat cube complexes.

\begin{thm}[Theorem~\ref{thm:fix_ccc}]\label{ccc intro} Let $G$ be a group acting essentially and non-elementarily on an irreducible finite dimensional \cat cube complex $X$. Then there exists a nonempty $G$-invariant closed convex subspace $X_0$ of $X$ such that $m(G_x)=1$ for every conjugation invariant mean $m$ on $G$ and every $x\in X_0$.
\end{thm}

Theorem \ref{thm:tree} corresponds precisely to the special case of Theorem \ref{ccc intro} where $X$ is one dimensional.

The essentiality and irreducibility assumptions in Theorem~\ref{ccc intro} can be removed at the cost of passing to a finite index subgroup; see Corollary~\ref{cor:non_irr_CC}. This allows us, for instance, to characterize when a graph product of groups is inner amenable in Theorem~\ref{thm:graph_prod}. Graph products generalize both direct products and free products of groups; examples of graph products of groups include right-angled Artin groups and right-angled Coxeter groups.

These examples, along with Theorems \ref{thm:amalgam} and \ref{thm:HNN}, illustrate that the most interesting applications of Theorem \ref{ccc intro} concern actions which are not necessarily proper. While Theorem~\ref{ccc intro} easily implies that groups acting properly and non-elementarily on finite dimensional \cat cube complexes are not inner amenable (see Corollary \ref{cor:CC_non-inner-amen}), this result can also be deduced from other results in the literature; see Remarks~\ref{rmk:DGO} and \ref{rmk:BIP}.

The proof of Theorem \ref{ccc intro} requires substantially more work than the one-dimensional case covered by Theorem \ref{thm:tree}. One novelty of the proof in the higher dimensional setting, which we now briefly describe, is that it makes use of the existence of certain idempotent conjugation invariant means.

The proof begins by observing that each conjugation invariant mean $m$ on $G$ must concentrate on the set of group elements which act elliptically (Proposition \ref{prop:ell_cat}). The next step uses a transversality argument (Proposition \ref{ccc}) to show that, for each half-space $\mathfrak{h}$ of $X$, there is a $m$-conull set of group elements which fix some point in $\mathfrak{h}$. At this point, the proof in the one dimensional setting is essentially complete, but the situation in higher dimensions becomes more complicated; after moving to a minimal convex subspace $X_0$ of $X$ and fixing some $x_0\in X$, we adapt an argument of Caprace and Sageev \cite{MR2827012} to our setting (Lemma \ref{lem:bounded}) to show that for each $x\in X_0$ the integral $\varphi (x) \defeq \int _G d(x,gx_0)^2 \, dm (g)$ is finite, and hence (using the \cat inequality) the function $x\mapsto \varphi (x)$ achieves a unique minimum at some point $z\in X_0$ (Lemma \ref{prop:almost-fix}). If the mean $m$ is idempotent under convolution, then we can easily deduce (using discreteness of $G$-orbits on $X$) that this point $z$ is fixed by a $m$-conull set of group elements, which would complete the proof. Fortunately, by using a simple stationarity argument (Lemma \ref{lem:stationary}) combined with Ellis's Lemma, we are able to reduce the proof of Theorem \ref{ccc intro} to the special case where the means $m$ under consideration are additionally assumed to be idempotent.

\subsection{Wreath products and the location lemma} We obtain a complete characterization of inner amenability for wreath products.

\begin{thm}[Theorem~\ref{thm:wreath}]\label{thm:wreath1}
Let $H\neq 1$ and $K$ be discrete groups, let $X$ be a set on which $K$ acts, and let $G\defeq H\wr_X K$ be the (restricted) wreath product. Then $G$ is inner amenable if and only if one of the following holds
\begin{enumerate}
\item The action $K\curvearrowright X$ admits an atomless $K$-invariant mean;
\item $H$ is inner amenable and the action $K\curvearrowright X$ has a finite orbit;
\item There is an atomless $K$-conjugation invariant mean $m$ on $K$ satisfying $m(K_x)=1$ for all $x\in X$.
\end{enumerate}
\end{thm}

One of the key ingredients to the proof of Theorem \ref{thm:wreath1} is Lemma \ref{lem:normal2}, which gives a way of locating various conjugation invariant means on a group, relative to some normal subgroup. Lemma \ref{lem:normal2} also leads to a complete characterization of when the commutator subgroup of an inner amenable group is itself inner amenable (see Corollary \ref{cor:iacommutator} and the paragraph preceding it). To state a further consequence, we first make a definition. If $m$ is a mean on a group $G$, then we define $\mathrm{ker}(m)\defeq  \{ g\in G \, : \, m (C_{G}(g)) = 1 \}$. It is easy to see that $\mathrm{ker}(m)$ is a subgroup of $G$, and if $m$ is conjugation-invariant, then $\mathrm{ker}(m)$ is a normal subgroup of $G$.

\begin{thm}[Corollary~\ref{cor:ncor1}]\label{thm:ncor1}
Let $G$ be an inner amenable group with no nontrivial finite normal subgroups, and let $N$ be a normal subgroup of $G$. Then either there exists an atomless conjugation invariant mean on $G$ which concentrates on $N$, or else $N\leq \mathrm{ker}(m)$ for every conjugation invariant mean $m$ on $G$.

Moreover, there exists an atomless conjugation invariant mean $m$ on $G$ with either $m([N,N])=1$ or $N\leq \mathrm{ker}(m)$.
\end{thm}


\begin{ack}
We would like to thank Yair Hartman for his many contributions to this project. We would also like to thank Amine Marrakchi for explaining to us another proof of Theorem \ref{thm:amalgam} and allowing us to include his argument in Remark \ref{rem:Amine}. RTD was supported in part by NSF grant DMS 1600904. BD was supported in part by French projects ANR-14-CE25-0004 GAMME and ANR-16-CE40-0022-01 AGIRA.\end{ack}

\section{Preliminaries}

For $G$ a group, $H\leq G$, and $g\in G$, we write
\[
g^H\defeq \{hgh^{-1}\mid h\in H\}.
\]
Let $m$ and $n$ be means on a group $G$. The {\bf convolution} of $m$ and $n$, denoted $m\ast n$, is the mean defined by
\[
(m\ast n )(A) \defeq  \int _{g\in G}n(g^{-1}A) \, dm (g)
\]
for $A\subseteq G$. We denote by $\check{m}$ the mean on $G$ defined by $\check{m}(A) \defeq  m (A^{-1})$ for $A\subseteq G$.

If either $m$ or $n$ is atomless, then $m\ast n$ is atomless as well; this is a direct computation under the assumption that $n$ is atomless, and it is a short exercise under the assumption that $m$ is atomless. Likewise, if $m$ and $n$ are both invariant under a subgroup $\mathcal{H}$ of $\mathrm{Aut}(G)$ then so are $m\ast n$ and $\check{m}$.

It follows readily from the definition that $m\mapsto m\ast n$ is continuous for the weak-$*$ topology but, the continuity of $n\mapsto m\ast n$ does not hold \emph{prima facie}.

\begin{lem}\label{lem:convcomp}
Let $m$ be a mean on $G$ and let $H$ be a subgroup of $G$. Then $(\check{m}\ast m)(H) \geq \sum _{gH \in G/H} m (gH)^2$. In particular, if there is some $g\in G$ such that $m(gH)>0$, then $(\check{m}\ast m)(H)> 0$.
\end{lem}

\begin{proof}
For any finite subset $F\subseteq G/H$, by finite additivity of $m$ we have
\begin{align*}
(\check{m}\ast m)(H) = \int _{G}m(xH)\, dm (x) & \geq \sum _{gH\in F} \int _{gH}m(xH)\, dm (x) \\
&= \sum _{gH\in F}\int _{gH} m(gH) \, dm (x)\\
&= \sum _{gH\in F}m(gH)^2 .
\end{align*}
Taking the supremum over all such $F$ proves the lemma.
\end{proof}

\begin{lem}\label{lem:stationary}
Let $H$ be a subgroup of a group $G$ and let $m$ and $n$ be means on $G$. Assume that $n(H)=1$. If either $n\ast m = n$ or $m\ast n = n$ then $m(H)=1$ as well.
\end{lem}

\begin{proof}
Suppose first that $n\ast m = n$. Then
\begin{align*}
1=n(H)=(n\ast m)(H) &= \int _G m(g^{-1}H)\, dn (g) \\
&= \int _H m(g^{-1}H)\, dn (g) = \int _H m(H) \, dn (g) = m(H),
\end{align*}
where the fourth and last equalities hold since $n$ concentrates on $H$.

Suppose now that $m\ast n = n$. Then $n(g^{-1}H)=0$ for all $g\not\in H$, hence
\begin{align*}
1=n(H)=(m\ast n )(H) &= \int _G n(g^{-1}H)\, dm (g) \\
&= \int _H n(H)\, dm (g) = \int _H 1 \, dm (g) =m(H) . \qedhere
\end{align*}
\end{proof}

\begin{prop}\label{prop:finind}
Let $G$ be a discrete group and let $\mathcal{H}$ be a subgroup of $\mathrm{Aut}(G)$. Let $m$ be an atomless $\mathcal{H}$-invariant mean on $G$. Then there exists another atomless $\mathcal{H}$-invariant mean $n$ on $G$ satisfying:
\begin{enumerate}
\item[(i)] If $K$ is a subgroup of $G$ with $m(K)=1$ then $n(K)=1$.
\item[(ii)] If $K$ is a subgroup of $G$ with $m(K)=1$ then $n(L)=1$ for every $\mathcal{H}$-invariant subgroup $L$ of $G$ with $|K:K\cap L|<\infty$.
\end{enumerate}
In particular, if $G$ is inner amenable, then there is an atomless conjugation invariant mean $m$ on $G$, such that $m(G_0)=1$ for every finite index subgroup $G_0$ of $G$.
\end{prop}

\begin{proof}
Let $\mathcal{C}_1$ be the collection of all subgroup $K$ of $G$ with $m(K)=1$, and let $\mathcal{C}$ be the collection of all $\mathcal{H}$-invariant subgroups $L$ of $G$ with $|K:K\cap L|<\infty$ for some $K\in \mathcal{C}_1$. Observe that $\mathcal{C}$ is a directed set under reverse inclusion: if $L_0 , L_1 \in \mathcal{C}$ are $\mathcal{H}$-invariant and $K_0,K_1\in \mathcal{C}_1$ are such that $|K_i:K_i\cap L_i|<\infty$, then $K_0\cap K_1\in \mathcal{C}_1$ and $|K_0\cap K_1 : K_0\cap K_1\cap L_0\cap L_1|<\infty$ since both $|K_0\cap K_1 : (K_0\cap L_0)\cap K_1 |\leq |K_0:K_0\cap L_0 | <\infty$ and $|(K_0\cap L_0)\cap K_1 : (K_0\cap L_0) \cap (K_1 \cap L_1)|\leq |K_1 : K_1\cap L_1|<\infty$.

Since $m$ is atomless and $\mathcal{H}$-invariant, so is $\check{m}\ast m$. If $K\in \mathcal{C}_1$ then we have $(\check{m}\ast m)(K)=\int _{K}m(kK)\, dm (k) = \int _{K}m(K)\, dm (k) = 1$. Moreover, if $K\in \mathcal{C}_1$ and $L$ is any subgroup of $G$ with $|K:K\cap L|<\infty$, then since $m$ is finitely additive there must be some $g_0\in G$ such that $m(g_0L)>0$, and hence $(\check{m}\ast m)(L)>0$ by Lemma \ref{lem:convcomp}. Thus, if $L$ is additionally $\mathcal{H}$-invariant (i.e., if $L\in \mathcal{C}$), then the normalized restriction, $n_L$, of $\check{m}\ast m$ to $L$, (defined by $n_L(A)\defeq  (\check{m}\ast m)(A\cap L)/(\check{m}\ast m)(L)$) is an $\mathcal{H}$-invariant atomless mean with $n_L(L)=1$ and $n_L(K)=1$ for all $K\in \mathcal{C}_1$. The assignment $L \mapsto n_L$ is then a net from the directed set $\mathcal{C}$ to the compact space of all atomless $\mathcal{H}$-invariant means on $G$. Any cluster point $n$ of this net then satisfies (i) and (ii).
\end{proof}

The following proposition will be improved significantly in Lemma \ref{lem:normal}, although it is important enough that we state it now.

\begin{prop}\label{prop:normvsquot}
Suppose that $G$ is inner amenable and let $N$ be a normal subgroup of $G$. Then either
\begin{enumerate}
\item There is an atomless $G$-conjugation invariant mean on $N$, or
\item $G/N$ is inner amenable.
\end{enumerate}
In particular, either $N$ is inner amenable or $G/N$ is inner amenable.
\end{prop}

\begin{proof}
Let $m$ be an atomless conjugation invariant mean on $G$. Let $p:G\rightarrow G/N$ denote the natural projection map. Then $p_*m$ is a conjugation invariant mean on $G/N$. If $p_*m$ is atomless, then $G/N$ is inner amenable so we are done. Otherwise, if $p_*m$ has an atom, then there is some $g\in G$ such that $m(gN)>0$. Then $(\check{m}\ast m)(N) >0$ by Lemma \ref{lem:convcomp}, so the normalized restriction of $\check{m}\ast m$ to $N$ is an atomless $G$-invariant mean on $N$.
\end{proof}

\begin{prop}\label{prop:finnorm}
Let $N$ be a finite normal subgroup of a group $G$. Then for every conjugation invariant mean $m_0$ on $G/N$, there is a conjugation invariant mean $m$ on $G$ which projects to $m_0$.
\end{prop}

\begin{proof}
Let $m_0$ be a conjugation invariant mean on $G/N$. Define the mean $m$ on $G$ by $m(A)\defeq \int _{gN\in G/N} |A\cap gN|/|N| \, dm_0 (gN)$. This clearly works.
\end{proof}

\begin{prop}\label{prop:coamen}
Let $H$ be a subgroup of a group $G$. Assume that
\begin{itemize}
\item[(a)] There is an atomless $H$-conjugation invariant mean $m_H$ on $G$,
\item[(b)] The action $G\curvearrowright G/H$ is amenable with $G$-invariant mean $m_{G/H}$.
\end{itemize}
Then $G$ is inner amenable, as witnessed by the atomless $G$-conjugation invariant mean
\[
m\defeq  \int _{gH\in G/H} gm_Hg^{-1} \, dm _{G/H}(gH)
\]
In particular, if $N$ is a normal subgroup of $G$ which is inner amenable, and if $G/N$ is amenable, then $G$ is inner amenable and, moreover, there is an atomless $G$-conjugation invariant mean $m$ on $G$ with $m(N)=1$.
\end{prop}

\begin{proof}
This is a straightforward computation.
\end{proof}

\begin{prop}\label{prop:almostisom}
Let $G$ be a group.
\begin{enumerate}
\item (Giordano, de la Harpe \cite{GH91}) Let $H$ be a finite index subgroup of $G$. Then $G$ is inner amenable if and only if $H$ is inner amenable.
\item Let $N$ be a finite normal subgroup of $G$. Then $G$ is inner amenable if and only if $G/N$ is inner amenable.
\end{enumerate}
\end{prop}

\begin{proof}
(1) follows from Propositions \ref{prop:finind} and \ref{prop:coamen}, and (2) follows from Proposition  \ref{prop:finnorm}.
\end{proof}

\section{Location lemma and wreath products}

\subsection{Lifting almost invariant probability measures}
Let $X$ and $Y$ be $G$-sets and let $\pi :X\rightarrow Y$ be a $G$-map from $X$ to $Y$. Let $\tilde{X}\defeq  X\otimes _{\pi} X = \{ (x_0,x_1)\in X^2 \, : \, \pi (x_0)=\pi (x_1)\}$, so that $\tilde{X}$ is a $G$-invariant subset of $X^2$ (for the diagonal $G$-action). Let $p\in \ell ^1(X)$ be a probability vector on $X$ (i.e., a nonnegative unit vector), and view $p$ as a probability measure on $X$. For $y\in Y$ let $p^y$ be the normalized restriction of $p$ to $\pi ^{-1}(y)$ (put $p^y=0$ if $p(\pi ^{-1}(y))=0$). Define $\tilde{p}\defeq  \sum _{y\in Y}p(\pi ^{-1}(y))(p^y\otimes p^y)$, so that $\tilde{p}$ is a probability vector on $\tilde{X}$.
\begin{lem}\label{lem:tildep}
For any $g\in G$ we have $\| g\tilde{p} - \tilde{p}\| _1 \leq 5\| gp - p \| _1$.
\end{lem}

\begin{proof}
For each $y\in Y$ let $q(y)\defeq  p(\pi ^{-1}(y))$ so that $\tilde{p}= \sum _{y\in Y}q(y)(p^y\otimes p^y)$. We write $gp^y$ for the translate of the function $p^y$ by $g$. We have
\begin{align*}
\| &g\tilde{p} - \tilde{p}\| _1 = \| \sum _{y\in Y} \big( q(y)(gp^y\otimes gp^y)- q(gy)(p^{gy}\otimes p^{gy})\big)\| _1 \\
&\leq \sum _{y\in Y} |q(y) - q(gy)| \| gp^y\otimes gp^y\| _1 + \sum _{y\in Y} q(gy)\| (gp^y\otimes gp^y) - (p^{gy}\otimes p^{gy}) \| _1
\end{align*}
Since $\| gp^y\otimes gp^y \| _1 \leq 1$ (possibly $p^y=0$), the first sum is bounded by $\| gp - p \| _1$. Let $Y_1 \defeq  \{ y \in Y\, : \, q(y)> 0 \}$. The second sum is bounded by
\begin{align*}
\sum _{y\in Y} q(gy) &\big( \| (gp^y- p^{gy})\otimes gp^y\| _1 + \| p^{gy}\otimes (gp^y- p^{gy}) \| _1 \big) \\
&\leq 2\sum _{y\in Y} q(gy)\| gp^y-p^{gy} \| _1 \\
&=2\sum _{y\in Y} \sum _{x\in \pi ^{-1}(gy)} |q(gy)p^y(g^{-1}x) - p(x)| \\
&\leq 2\| gp-p \| _1 + 2 \sum _{y\in Y} \sum _{x\in \pi ^{-1}(gy)} |q(gy)p^y(g^{-1}x) - p(g^{-1}x)| \\
&= 2\| g p - p \| _1 + 2 \sum _{y\in Y_1}|\frac{q(gy)}{q(y)} - 1| \sum _{x\in\pi ^{-1}(gy)}p(g^{-1}x) \\
&\leq 2\| gp - p \| _1 + 2\sum _{y\in Y} |q(gy)-q(y)|\\
&\leq 4 \| gp-p \| _1 .\qedhere
\end{align*}
\end{proof}

\subsection{Conjugation invariant means on normal subgroups}
In what follows, for a subgroup $M$ of a group $G$, and a nonempty subset $S$ of $G$, we define
\[
C_{G/M}(S) \defeq  \{ g\in G \, : \, g(sM)g^{-1}=sM \text{ for  all }s\in S \} .
\]
For $h\in G$ we write $C_{G/M}(h)$ for $C_{G/M}(\{ h \} )$. If $M$ is the trivial subgroup, we simply denote $C_G(S)$ for the centralizer of $S$ in $G$.%
\begin{prop}\label{prop:CGM}
Let $M$ be a subgroup of $G$, let $S$ be a nonempty subset of $G$, and let $\langle S\rangle$ be the subgroup generated by $S$. Then:
\begin{enumerate}
\item $C_{G/M}(S)$ is a subgroup of $G$ contained in the normalizer, $N_G(M)$, of $M$ in $G$.
\item Suppose that $S\subseteq N_G(M)$. Then $C_{G/M}(\langle S\rangle)=C_{G/M}(S)$.
\item Suppose that $M$ and $S$ are finite and $S\subseteq N_G(M)$. Then $C_G(S)\cap C_G(M)$ is a finite index subgroup of $C_{G/M}(S)$.
\end{enumerate}
\end{prop}

\begin{proof}
(1): $C_{G/M}(S)$ is clearly a group. To see that $C_{G/M}(S)$ is contained in $N_G(M)$, observe that for $g\in C_{G/M}(S)$ and $s\in S$ we have (a) $M=s^{-1}g^{-1}sMg$, and (b) $s^{-1}gsg^{-1} \in M$. By (b) we have $g^{-1}M = s^{-1}g^{-1}sM$, and applying this to (a) we see that $M=s^{-1}g^{-1}sMg = g^{-1}Mg$, and hence $g\in N_G(M)$.

(2): This is clear.

(3): Since $M$ is finite, for each $s\in S$ the group $C_G(M)\cap C_{G/M}(s)$ has finite index in $C_{G/M}(s)$, and hence the group $C_G(sM)$ has finite index in $C_{G/M}(s)$, being the kernel of the homomorphism $(C_G(M)\cap C_{G/M}(s))\rightarrow M$, $g\mapsto [s,g]$. Therefore, since $S$ is finite, $C_G(S)\cap C_G(M)=\bigcap _{s\in S}C_G(sM)$ has finite index in $C_{G/M}(S)=\bigcap _{s\in S}C_{G/M}(s)$.
\end{proof}
%
%

Let $m$ be a mean on a group $G$. We define
\[
\mathrm{ker}(m)\defeq  \{ g\in G\, : \, m (C_G(g)) = 1 \} .
\]
It is easy to see that $\mathrm{ker}(m)$ is a subgroup of $G$. If $m$ is invaraint under conjugation by a subgroup $K$ of $G$, then $\mathrm{ker}(m)$ is normalized by $K$. Observe that if $H_0$ is any finitely generated subgroup of $\mathrm{ker}(m)$, then $m(C_G(H_0)) = 1$. Thus, if the mean $m$ is atomless then $C_G(H_0)$ is infinite for every finitely generated subgroup $H_0$ of $\mathrm{ker}(m)$.

Given two elements $h$ and $k$ in a group $G$, we denote their commutator by $[h,k]\defeq hkh^{-1}k^{-1}$. If $H$ and $K$ are subgroups of $G$ then we define $[H,K]$ to be the subgroup
\[
[H,K]\defeq  \langle \{ [h,k] \, : \, h\in H , \ k\in K \} \rangle .
\]
Clearly $[H,K]=[K,H]$. Note that the group $[H,K]$ is normalized by both $H$ and $K$. To see this, observe that if $[h,k]$ is any generator for $[H,K]$, where $h\in H$ and $k\in K$, then for any $h_0 \in H$ we have $h_0[h,k]h_0^{-1}=[h_0h,k][h_0,k]^{-1} \in [H,K]$. Hence $[H,K]$ is normalized by $H$, and by symmetry it is also normalized by $K$.

The following lemma, along with the more general Lemma \ref{lem:normal2}, is one of our main tools for understanding the ``location'' of conjugation invariant means on a group.

\begin{lem}\label{lem:normal}
Let $N$, $G_0$ and $K$ be subgroups of a group $G$, with $N$ normalized by $K$. Assume that there is no nontrivial finite subgroup of $G_0\cap [G_0,K\cap N]$ which is normalized by $K$. Then at least one of the following holds:
\begin{enumerate}
\item[$(1)$] There is an atomless $K$-conjugation invariant mean on $G$ which concentrates on $G_0\cap [G_0, K\cap N]$,
\item[$(2)$] For every $K$-conjugation invariant mean $m$ on $G$ which concentrates on $G_0$, we have $K\cap N\leq \mathrm{ker}(m)$.
\end{enumerate}
\end{lem}

\begin{rmk}
The assumption that there is no nontrivial finite subgroup of $G_0\cap [G_0,K\cap N]$ which is normalized by $K$ can be removed at the expense of making the conclusion of the lemma a bit messier to state. See Lemma \ref{lem:normal2} below.
\end{rmk}

\begin{proof}[Proof of Lemma \ref{lem:normal}]
Assume that (1) fails and we will show that (2) holds. Fix then a mean $m$ on $G$ with $m(G_0)=1$ and which is invariant under conjugation by $K$. We must show that $m(C_{G}(h))=1$ for each $h\in K\cap  N$.

We first claim that there is no $g\in G - \{ 1_G \}$ with $g^K$ finite and contained in $G_0\cap [G_0,K\cap N]$. For suppose otherwise. If the group $\langle g^K \rangle$ were finite, then it would be a nontrivial finite subgroup of $G_0\cap [G_0,K\cap N]$ which is normalized by $K$, a contradiction. So the group $\langle g^K\rangle$ would have to be infinite. But every element of this group has a finite orbit under conjugation by $K$, so since $\langle g^K\rangle$ is infinite we can find an atomless $K$-conjugation invariant mean which concentrates on $\langle g^K\rangle \leq G_0\cap [G_0,K\cap N]$, which contradicts the hypothesis that (1) fails.

Let $X=G$ and let $\alpha : K\curvearrowright X$ denote the conjugation action of $K$ on $X$, $\alpha (k)x \defeq  kxk^{-1}$. By the Hahn-Banach Theorem we may find a net $(p_i)_{i\in I}$ of probability vectors on $X$, which weak${}^*$-converges to $m$ in $\ell ^{\infty }(X)^*$, and satisfies $\| \alpha (k)(p_i) - p_i \| _1\rightarrow 0$ for all $k\in K$. Since $m$ concentrates on $G_0$, we may additionally assume that each $p_i$ concentrates on $G_0$. Let $Y$ denote the set of all orbits of $\alpha (K\cap N)$ on $X$, and let $\pi : X\rightarrow Y$ denote the projection map, $\pi (x) \defeq \alpha  (K\cap N)x$. Since $K$ normalizes $K\cap N$, $K$ naturally acts on $Y$ so that $\pi$ is a $K$-equivariant map. Let
\[
\tilde{X}\defeq X\otimes _{\pi}X = \{ (x_0,x_1)\in X\times X \, : \, \pi (x_0)=\pi (x_1) \} ,
\]
and let $\tilde{\alpha}:K\curvearrowright \tilde{X}$ denote the diagonal action of $K$, i.e., $\tilde{\alpha}(k)(x_0,x_1)\defeq (\alpha (k)(x_0), \alpha (k)(x_1))= (kx_0k^{-1},kx_1k^{-1})$. For each probability vector $p$ on $X$ and each $y\in Y$, let $p^y$ be the normalized restriction of $p$ to $\pi ^{-1}(y)$ (where we put $p^y=0$ if $p(\pi^{-1}(y))=0$), and define the probability vector $\tilde{p}$ on $\tilde{X}$ by $\tilde{p}\defeq \sum _{y\in Y}p(\pi^{-1}(y)) p^y\otimes p^y$. By Lemma \ref{lem:tildep}, for each $k\in K$ we have $\| \tilde{\alpha}(k)\tilde{p}_i - \tilde{p}_i \| _1 \rightarrow 0$. Thus, after moving to a subnet of $(p_i)_{i\in I}$ if necessary, we may assume without loss of generality that the net $(\tilde{p}_i)_{i\in I}$ weak${}^*$-converges in $\ell ^{\infty}(\tilde{X})^*$ to a mean $\tilde{m}$ on $\tilde{X}$ which is invariant under $\tilde{\alpha}(K)$. Then $\tilde{m}$ concentrates on $\tilde{X} \cap (G_0\times G_0)$ since each $\tilde{p}_i$ does.

For each $(x_0,x_1)\in \tilde{X}\cap (G_0\times G_0)$ we have $x_0,x_1 \in G_0$ and $x_1 = hx_0h^{-1}$ for some $h\in K\cap N$, hence $x_0x_1^{-1} = x_0hx_0^{-1}h^{-1} \in G_0\cap [G_0,K\cap N]$. We let $\varphi : \tilde{X}\rightarrow X$ be the map $\varphi (x_0,x_1)\defeq x_0x_1^{-1}$. The map $\varphi$ is $K$-equivariant, i.e., $\varphi\circ \tilde{\alpha}(k)=\alpha (k)\circ \varphi$, hence the pushforward $\varphi _*\tilde{m}$, of $\tilde{m}$ under $\varphi$, is a $K$-conjugation invariant mean on $X=G$ satisfying $(\varphi _*\tilde{m} ) (G_0\cap [G_0,K\cap N])=1$. By our assumption, the mean $\varphi _*\tilde{m}$ must by purely atomic and, being $K$-invariant, must therefore concentrates on the collection of finite orbits of $\alpha (K)$ which are contained in $G_0\cap [G_0,K\cap N]$. As we saw above, the only such orbit is the trivial orbit of the identity element. This means that $\varphi _*\tilde{m}$ is the point mass at the identity element of $G$, and hence $\tilde{m}(\triangle _X ) =1$, where $\triangle _X \defeq  \{ (x,x)\, : \, x\in X \}\subseteq \tilde{X}$.

For each $i\in I$ let $q_i = \pi _*p_i$. For $i\in I$ and $y\in Y$, let $s_i (y) \defeq  \sup _{x\in \pi ^{-1}(y)}p_i^y(x)=\max _{x\in \pi ^{-1}(y)} p_i^y (x)$. Since $(\tilde{p}_i )_{i\in I}$ weak${}^*$-converges to $\tilde{m}$, we have $1= \tilde{m} (\triangle _X) = \lim _i \tilde{p}_i (\triangle _X )$, and so
\begin{align}
\nonumber 1 = \lim _i \tilde{p}_i (\triangle _X) &= \lim _i \sum _{y\in Y}q_i(y) \sum _{x\in \pi ^{-1}(y)} p_i^y(x)^2 \\
\nonumber &\leq \lim _i \sum _{y \in Y} q_i(y)s_i(y)\sum _{x\in \pi ^{-1}(y)}p_{i}^y (x) \\
\label{eqn:siO} &= \lim _i \int _{Y} s_i(y) \, dq_i (y).
\end{align}
For each $i\in I$ and $y\in Y$ choose some $x_i^y \in \pi ^{-1}(y)$ with $p_{i}^y (x_{i}^y)=s_i(y)$. Let $r\defeq \tfrac{3}{4}$ (although any number strictly between $\tfrac{1}{2}$ and $1$ will do), and define $Y_i \defeq  \{ y\in Y\, : s_i (y) >r \}$ and $X_i \defeq  \{ x_{i}^y \, : \, y\in Y \}$. Since $0\leq s_i (y)\leq 1$ we have
\begin{align*}
\int s_i \, dq_i = \int _{Y_i} s_i \, dq_i + \int _{Y\setminus Y_i} s_i \, dq_i &\leq q_i (Y_i) + r(1- q_i(Y_i)) \\
&= r + (1-r)q_i (Y_i ),
\end{align*}
and hence, by \eqref{eqn:siO},
\begin{equation}\label{eqn:Yi}
\lim _i p_i (\pi ^{-1}(Y_i)) = \lim _i q_i(Y_i)\geq \lim _i [\int s_i \, dq_i - r]/(1-r)  = 1 .
\end{equation}
In addition,
\begin{equation}\label{eqn:Xi}
\lim _i p_i (X_i) = \lim _i \sum _{y\in Y} p_{i}(x_{i}^y) = \lim _i \int _{Y}s_i (y) \, dq_i (y)  = 1 .
\end{equation}
Fix now any $h\in K\cap N$. For each $x=x_i^y \in (X_i\cap \pi ^{-1}(Y_i))\setminus C_G(h)$, we have $\alpha (h)x \neq x$, hence $p_i^y(\alpha (h)x)<1-p_i^y(x)<1-r$, and
\begin{align*}
p_i(x)(2r-1)\leq q_i(y) (r - (1-r))&\leq q_i(y)(p_{i}^y(x) - p_{i}^y(\alpha (h)x)) \\
&= |p_i(x) - p_i(\alpha (h)x)|,
\end{align*}
hence $p_i ((X_i\cap \pi ^{-1}(Y_i) )\setminus C_G(h)) \leq \| p_i - \alpha (h)p_i \| _1/(2r-1)$. By \eqref{eqn:Yi} and \eqref{eqn:Xi}, we therefore conclude that
\begin{align*}
m(X\setminus C_G(h)) = \lim _i p_i (X \setminus C_G(h)) &= \lim _i p_i ((X_i\cap \pi ^{-1}(Y_i) )\setminus C_G(h)) \\
& \leq \lim _i \frac{\| p_i - \alpha (h)p_i \|}{2r-1}\\
&= 0 .
\end{align*}
Therefore $m(C_G(h)) =1$ for all $h\in K\cap N$, as was to be shown. This finishes the proof
\end{proof}

\begin{lem}[Location lemma]\label{lem:normal2}
Let $N$, $G_0$ and $K$ be subgroups of a group $G$, with $N$ normalized by $K$. Let $P$ be a subgroup of $G$ which is normalized by $G_0$ and contains $G_0\cap [G_0,K\cap N]$ (e.g., $P=G_0\cap [G_0,K\cap N]$), and let $M$ be defined by
\[
M\defeq  \{ g\in G \, : \, g^K\text{ is finite and contained in }P  \} ,
\]
so that $M$ is a subgroup of $P$ which is normalized by $K$. 
Then at least one of the following holds:
\begin{enumerate}
\item[$(1)$] There is an atomless $K$-conjugation invariant mean on $G$ which concentrates on $P$,
\item[$(2)$] The group $M$ is finite, and for every $K$-conjugation invariant mean $m$ on $G$ which concentrates on $MG_0$, we have $m(C_{G/M}(h))=1$ for every $h\in MK\cap N$.
\end{enumerate}
\end{lem}

\begin{proof}[Proof of Lemma \ref{lem:normal2}]
Assume (1) fails and we will show that (2) holds. Suppose toward a contradiction that $M$ is infinite. Then we could finite a sequence $C_0,C_1,\dots$, of distinct finite $K$-orbits for the conjugation action of $K$ on $M$. If we let $m_{C_n}$ denote normalized counting measure on $C_n$, then any weak${}^*$-cluster point of $m_{C_0},m_{C_1}\dots$, is an atomless $K$-conjugation invariant mean on $G$ which concentrates on $P$, a contradiction. Therefore, $M$ must be finite. Let $N_G(M)$ denote the normalizer of $M$ in $G$. We next establish the following:
\begin{center}
$(*)$ If $m$ is any $K$-conjugation invariant mean on $G$ which concentrates on $MG_0$, then $m(N_{G}(M))=1$.
\end{center}
To see this, let $m$ be a $K$-conjugation invariant mean on $G$ with $m(MG_0)=1$. Since $P$ is normalized by $G_0$ and $M\leq P$, we obtain a mean $n$ on $G$, concentrating on $P$, defined by $n \defeq  \int _{g\in MG_0} n_{g^{-1}Mg}\, dm (g)$, where $n_{g^{-1}Mg}$ denotes normalized counting measure on the finite group $g^{-1}Mg\leq P$ for $g\in MG_0$. Since $m$ is invariant under conjugation by $K$ and $M$ is normalized by $K$, the mean $n$ is invariant under conjugation by $K$. Since (1) fails, the mean $n$ must be purely atomic and hence we must have $n(M)=1$, which implies that $m(N_G(M))=1$. This proves $(*)$.

Let $G'$ be the quotient group $G'\defeq  N_G(M)/M$, and let $\mathrm{proj}:N_G(M) \rightarrow G'$ be the projection map. Define $K'\defeq \mathrm{proj}(K)$, $G_0' \defeq  \mathrm{proj}(G_0 \cap N_{G}(M))$, $N'\defeq \mathrm{proj}(N\cap N_G(M))$, and $P' \defeq  \mathrm{proj}(P\cap N_G(M))\geq G_0' \cap [G_0',K'\cap N' ]$. Observe that the group $M' \defeq  \{ x \in G' \, : \, x^{K'} \text{ is finite and contained in }P' \}$ is trivial: every orbit of the conjugation action $K'\curvearrowright G'$, is the image under $\mathrm{proj}$ of an orbit of the conjugation action $K\curvearrowright N_G(M)$, and since $M$ is finite this means that finite orbits of $K'\curvearrowright G'$ are images of finite orbits of $K\curvearrowright N_G(M)$, hence $M'$ is trivial. In particular, there is no nonidentity $x \in G' - \{ 1 _{G'} \}$ such that $x^{K'}$ is finite and contained in $G_0' \cap [G_0',K'\cap N']$.

This establishes all of the hypotheses of Lemma \ref{lem:normal} for the groups $G',N',G_0'$, and $K'$ in place of $G,N,G_0$, and $K$ respectively. Observe that (1) of Lemma \ref{lem:normal} fails for these groups: if $m'$ were an atomless $K'$-conjugation invariant mean on $G'$ which concentrates on $G_0'\cap  [G_0',K'\cap N']$, then we would obtain an atomless $K$-conjugation invariant mean $m$ on $G$ concentrating on $P$, defined by $m (A) \defeq  \int _{gM\in P'} \tfrac{|A\cap gM|}{|M|} \, dm' (gM )$, a contradiction. Thus, (2) of Lemma \ref{lem:normal} must hold, i.e., for every $K'$-conjugation invariant mean $m'$ on $G'$ which concentrates on $G_0'$, we have $m' (C_{G'}(h')) = 1$ for every $h'\in K'\cap N'= \mathrm{proj}(MK\cap N)$. In particular, by $(*)$, this applies to all means $m'$ which are the projection, $m' \defeq  \mathrm{proj}_*m$, of some $K$-conjugation invariant mean $m$ on $G$ which concentrates on $MG_0$. Therefore, for all such means $m$ we have $m (C_{G/M}(h))=1$, as was to be shown.
\end{proof}

As a Corollary to Lemma \ref{lem:normal} we obtain:

\begin{cor}\label{cor:ncor1}
Let $N$ be a normal subgroup of a group $G$ and assume that there is no nontrivial finite normal subgroup of $G$ contained in $[G,N]$. Then at least one of the following holds:
\begin{enumerate}
\item[$(A.1)$] There is an atomless $G$-conjugation invariant mean which concentrates on $[N,N]$.
\item[$(A.2)$] There is an atomless $G$-conjugation invariant mean which concentrates on $[G,N]$, and for every $G$-conjugation invariant mean $m$ on $G$ which concentrates on $N$, we have that $N\leq \mathrm{ker}(m)$.
\item[$(A.3)$] For every $G$-conjugation invariant mean $m$ on $G$ we have $N\leq \mathrm{ker}(m)$.
\end{enumerate}
In particular, if $G$ is inner amenable then there exists an atomless $G$-conjugation invariant mean $m$ on $G$ with either $m([N,N])=1$ or $N\leq \mathrm{ker}(m)$.
\end{cor}

\begin{proof}
Assume that $(A.1)$ and $(A.2)$ fail and we will prove that $(A.3)$ holds. Apply Lemma \ref{lem:normal} with $G_0=N$ and $K=G$. The assumption that $(A.3)$ fails implies that alternative (1) of that lemma fails, and hence alternative (2) holds, i.e., for every $G$-conjugation invariant mean $m$ on $G$ which concentrates on $N$ we have $N\leq \mathrm{ker}(m)$. But this means that the second part of $(A.2)$ holds, so the assumption that $(A.2)$ fails then implies that there is no atomless $G$-conjugation invariant mean which concentrates on $[G,N]$. Applying Lemma \ref{lem:normal} again, but this time using $G_0=K=G$, this means that (1) of that lemma fails, and hence (2) must hold, i.e., for every $G$-conjugation invariant mean $m$ on $G$ we have $N\leq \mathrm{ker}(m)$, which is precisely $(A.3)$.
\end{proof}

\begin{cor}\label{cor:ncor2}
Let $G$  be an inner amenable group and let $N$ be a normal finitely generated subgroup of $G$. Then there is an atomless $G$-conjugation invariant mean $m$ on $G$ with either $m([N,N])=1$ or $m(C_G(N))=1$.
\end{cor}

\begin{proof}
Note that if $[G,N]$ contains no nontrivial finite normal subgroup of $G$ then this follows immediately from Corollary \ref{cor:ncor1}. In the general case, we will argue similarly, but apply Lemma \ref{lem:normal2} instead of Lemma \ref{lem:normal}. Suppose that there is no atomless $G$-conjugation invariant mean which concentrates on $[N,N]$. Then (1) of Lemma \ref{lem:normal2} fails (applied to $K=G$, $G_0=N$, and $P=[N,N]$), so the group $M$, of all elements of $[N,N]$ with finite $G$-conjugacy classes, is finite, and every $G$-conjugation invariant mean $m$ which concentrates on $N$ must satisfy $m(C_{G/M}(N))=1$ (since $N$ is finitely generated). We now consider cases.

Case 1: There exists an atomless $G$-conjugation invariant mean $m$ which concentrates on $N$. In this case, by what we already showed we have $m(C_{G/M}(N))=1$. Since $M$ is finite and $N$ is finitely generated, the group $C_G(N)$ has finite index in $C_{G/M}(N)$. Therefore, by Proposition \ref{prop:finind}, we can find an atomless $G$-conjugation invariant mean on $G$ which concentrates on $C_G(N)$, as was to be shown.

Case 2: There does not exist an atomless $G$-conjugation invariant mean $m$ which concentrates on $N$. In this case, applying Lemma \ref{lem:normal2} again, but this time with $K=G_0=G$ and $P=N$, we see that (1) of that lemma fails, and hence (2) holds, i.e., the group $M_1$, of all elements of $N$ with finite $G$-conjugacy class, is finite, and every $G$-conjugation invariant mean $m$ on $G$ satisfies $m(C_{G/M_1}(N))=1$ (once again, since $N$ is finitely generated). Since by assumption $G$ is inner amenable, we can find an atomless $G$-conjugation invariant mean $m_1$ on $G$, which must necessarily satisfy $m_1(C_{G/M_1}(N))=1$. We now argue as in Case 1. Since $M_1$ is finite and $N$ is finitely generated, the group $C_G(N)$ has finite index in $C_{G/M_1}(N)$, so we can apply Proposition \ref{prop:finind} to find an atomless $G$-conjugation invariant mean satisfying $m(C_G(N))=1$.
\end{proof}

Taking $N=K=G_0=G$ and $P=[G,G]$ in Lemma \ref{lem:normal2} shows that, if $G$ is an inner amenable group, then either the commutator subgroup of $G$ is inner amenable, or else there exists a finite normal subgroup $M$ of $G$ such that $G/M$ admits an asymptotically central sequence (i.e., the centralizer of every finite subset of $G/M$ is infinite). If $G$ is additionally finitely generated and has no nontrivial finite normal subgroups, then we obtain:

\begin{cor}\label{cor:iacommutator}
Let $G$ be a group. Assume that $G$ is finitely generated and contains no nontrivial finite normal subgroups. If $G$ is inner amenable then exactly one of the following holds:
\begin{enumerate}
\item The commutator subgroup $[G,G]$ of $G$ is inner amenable,
\item The center, $Z(G)$, of $G$, is infinite, $[G,G]\cap Z(G)=1$, and every $[G,G]$-conjugation invariant mean on $G$ concentrates on $Z(G)$. Moreover, there is a finitely generated subgroup $H$ of $[G,G]$ such that $C_G(H)=Z(G)$.
\end{enumerate}
\end{cor}

\emph{Remark:} In alternative (2), $Z(G)$ must be isomorphic to $\Zb ^n$ for some $n \geq 1$: since $G$ has no nontrivial finite normal subgroups, it follows that $Z(G)$ is torsion free, and since $G$ is finitely generated with $[G,G]\cap Z(G)=1$, it follows that $Z(G)$ is isomorphic to a torsionfree subgroup of the finitely generated abelian group $G/[G,G]$.

\begin{proof}
The two alternatives are indeed mutually exclusive: if (1) holds, then the collection $C$, of atomless $[G,G]$-conjugation invariant means on $[G,G]$ is a nonempty compact convex set on which $G$ acts by conjugation, with $[G,G]$ acting trivially. Since $G/[G,G]$ is abelian, by the Markov-Kakutani fixed point theorem, there is a fixed point $m\in C$, which corresponds to an atomless $G$-conjugation invariant mean with $m([G,G])=1$. Therefore, (2) cannot hold since it would imply that $m(Z(G))=1$ and hence $m(\{ 1 \} )=m([G,G]\cap Z(G))=1$, contradicting that $m$ is atomless.

Assume now (1) fails and we will show that (2) holds. Since $G$ is finitely generated, applying Lemma \ref{lem:normal} (with $N=K=G_0=G$) shows that $m(Z(G))=1$ for every $G$-conjugation invariant mean $m$ on $G$. Suppose toward a contradiction that there were some $[G,G]$-conjugation invariant mean $n_0$ on $G$ with $n_0(Z(G))=r<1$. Then the set $D$ of all $[G,G]$-conjugation invariant means $n$ on $G$ with $n(Z(G))=r$ is a nonempty compact convex set on which $G$ acts by conjugation, with $[G,G]$ acting trivially. Applying the Markov-Kakutani fixed point theorem once more yields a mean $m\in D$ which is invariant under conjugation by $G$, a contradiction.

We claim that there is some finitely generated subgroup $H_0$ of $[G,G]$ such that $[G,G]\cap C_G(H_0) =1$ (and hence $[G,G]\cap Z(G)=1$). Otherwise, for each finitely generated subgroup $H$ of $[G,G]$ we can find some $x_H\in [G,G]\cap C_G(H)$ with $x_H\neq 1$. The collection of finitely generated subgroups of $[G,G]$ is a directed set under inclusion and, taking any weak${}^*$ cluster point of the net of point masses, $(\delta _{x_H})$, in the space of means on $[G,G]$, we obtain a $[G,G]$-conjugation invariant mean $m$ on $[G,G]$ which is not the point mass at the identity element. Since (1) does not hold, the group $M$, of all elements of $[G,G]$ with finite conjugacy class, must be finite and $m$ must concentrate on $M$. Since $M$ is characteristic in $[G,G]$, it is normal in $G$, and by hypothesis $G$ has no nontrivial finite normal subgroups, hence $M=1$. This contradicts that $m$ is not the point mass at the identity element.

Similarly, we claim that there must be a finitely generated subgroup $H_1$ of $[G,G]$ such that $C_G(H_1)=Z(G)$. Otherwise, for each finitely generated subgroup $H$ of $[G,G]$ we can find some $y_H\in C_G(H)$ with $y_H\not\in Z(G)$. By taking any weak${}^*$ cluster point of the net of point masses, $(\delta _{y_H})$, in the space of means on $G$, we obtain a $[G,G]$-conjugation invariant mean $n$ on $G$ with $n(Z(G))=0$, which we already showed is impossible.
\end{proof}

\subsection{Wreath products}\label{subsec:wreath}

\begin{thm}\label{thm:wreath}
Let $K$ and $H$ be groups, with $H\neq 1$, let $X$ be a set on which $K$ acts, and let $G\defeq H\wr_XK= (\bigoplus _X H) \rtimes K$ be the (restricted) generalized wreath product. Then $G$ is inner amenable if and only if one of the following holds
\begin{enumerate}
\item The action $K\curvearrowright X$ admits an atomless $K$-invariant mean,
\item $H$ is inner amenable and the action $K\curvearrowright X$ has a finite orbit,
\item There is an atomless $K$-conjugation invariant mean $m$ on $K$ satisfying $m(K_x)=1$ for all $x\in X$, where $K_x$ denotes the stabilizer of $x$ in $K$.
\end{enumerate}
\end{thm}

\emph{Remark:} It follows from the proof that either (1) or (2) holds if and only if there is an atomless conjugation invariant mean on $G$ which concentrates on $\bigoplus _X H$, and that (3) holds if and only if there is an atomless conjugation invariant mean on $G$ which concentrates on $K$.

\begin{proof}
Let $N\defeq \bigoplus _X H$, so that elements of $N$ are functions $z:X\rightarrow H$ whose support, $\mathrm{supp}(z)\defeq  \{ x\in X\, : \, z(x)\neq 1_H \}$, is finite. Then $K$ acts on $N$ via $(k\cdot z)(x)\defeq  z(k^{-1}\cdot x )$, and we identify $N$ and $K$ with subgroups of $G$ so that $G$ is the internal semidirect product $G=N\rtimes K$ and $kzk^{-1}=k\cdot z$ for all $k\in K$, $z\in N$.

Suppose first that (1) holds. Let $m$ be an atomless $K$-invariant mean on $X$. Fix some $h\in H$, $h\neq 1$, and for each $x\in X$ let $z_x \in N$ be the function $z_x(x)=h$ and $z_x(y)=1_H$ if $y\neq x$. Define the mean $\tilde{m}$ on $N$ to be the pushforward of $m$ under the map $x\mapsto z_x$. Since $z_{k\cdot x}=kz_xk^{-1}$, the mean $\tilde{m}$ is invariant under conjugation by $K$. Since $m$ is atomless, $\tilde{m}$ is atomless, and for any $z\in N$ we have $\tilde{m}(C_N(z))\geq m(X\setminus \mathrm{supp}(z))=1$ since $\mathrm{supp}(z)$ is finite. Therefore, $\tilde{m}$ is invariant under conjugation by $N$ and $K$, hence by all of $G$, so $G$ is inner amenable.

Suppose next that (2) holds. Let $X_0\subseteq X$ be a finite orbit of the action $K\curvearrowright X$, and let $K_0=\bigcap _{x\in X_0}K_x$, so that $K_0$ has finite index in $K$. Then the subgroup $N_0\defeq  \{ z\in N \, : \, \mathrm{supp}(z)\subseteq X_0 \}$ is normal in $G$ and isomorphic to the finite direct sum $N_0\cong \bigoplus _{X_0}H$. Since $H$ is inner amenable, $N_0$ is inner amenable, and hence $N_0C_G(N_0)$ is inner amenable. Since $N_0C_G(N_0)$ contains $NK_0$, it is of finite index in $G$, hence $G$ is inner amenable.

Suppose that (3) holds. For any $z\in N$ we have $C_K(z)=\bigcap _{x\in \mathrm{supp}(z)}K_x$ hence $m(C_K(z)) = 1$. Therefore, in addition to being $K$-conjugation invariant, the mean $m$ is $N$-conjugation invariant, hence it is $G$-conjugation invariant, hence $G$ is inner amenable. 

Assume now that $G$ is inner amenable and we will show that (1), (2), or (3) holds. Let $M$ denote the group of all elements of $N$ having finite $G$-conjugacy class, and let $M_H$ denote the group of elements of $H$ having finite $H$-conjugacy class. Notice that if $z\in M$, then $\mathrm{supp}(z)$ is contained in a finite $K$-invariant set, and $z(x)\in M_H$ for all $x\in X$. Therefore, if $M$ is infinite then either $K$ has infinitely many finite orbits on $X$, in which case (1) holds, or $K$ has a (nonzero) finite number of finite orbit on $X$ and $M_H$ is infinite, in which case (2) holds. We may therefore assume without loss of generality that $M$ is finite.

Let $\pi : G\rightarrow N$ be the map $\pi (zk)=z$ for $z\in N$, $k\in K$. This map is equivariant for the conjugation actions of $K$ on $G$ and $N$ respectively. If $m$ is any mean on $G$ with $m(K)<1$ then we let $m_0$ denote the normalized restriction of $m$ to $G-K$ and we define the mean $\varphi (m)$ on $X$ by
\[
\varphi (m)(A) \defeq  \int _{z\in N} \frac{|A\cap \mathrm{supp}(z)|}{|\mathrm{supp}(z)|} \, d\pi _*m_0(z)
\]
for $A\subseteq X$. It is clear that if $m$ is $K$-conjugation invariant, then $\varphi (m)$ is invariant under the action of $K$ on $X$. We will break the rest of the proof into three cases:

\begin{itemize}
\item[{\bf C1}] There exists a $K$-conjugation invariant mean $m$ on $G$ with $m(K)<1$ such that $\varphi (m)$ is not supported on a finite subset of $X$.
\item[{\bf C2}] There exists an atomless $G$-conjugation invariant mean $m$ on $G$, with $m(MK)<1$ and $m(G_0)=1$ for every finite index subgroup $G_0$ of $G$, and such that $\varphi (m)$ is supported on a finite subset of $X$.
\item[{\bf C3}] There exists an atomless $G$-conjugation invariant mean $m$ on $G$ such that $m(C_{G/M}(z))=1$ for all $z\in N$, and $m(MC_K(M))=1$.
\end{itemize}

We will show that {\bf C1} implies (1), {\bf C2} implies (2), and {\bf C3} imples (3). Let us first assume that these three implications hold and finish the proof. Suppose first that there exists an atomless $G$-conjugation invariant mean $m$ on $G$ which concentrates on $N$. By Proposition \ref{prop:finind} we may assume that $m(G_0)=1$ for every finite index subgroup of $G$. Since $m$ is atomless and $M$ is finite we have $m(MK)= m(N\cap MK)=m(M)=0$. If $\varphi (m)$ is not supported on a finite subset of $X$ then {\bf C1} holds, hence (1) holds, and if $\varphi (m)$ is supported on a finite subset of $X$ then {\bf C2} holds, hence (2) holds. We may therefore assume that there is no atomless $G$-conjugation invariant mean which concentrates on $N$. Then by Lemma \ref{lem:normal2}, for every $G$-conjugation invariant mean $m$ on $G$ we must have $m(C_{G/M}(z))=1$ for all $z\in N$.  Since $G$ is inner amenable we may find an atomless $G$-conjugation invariant mean $m$ on $G$, and by Proposition \ref{prop:finind} we may assume that $m(G_0)=1$ for every finite index subgroup $G_0$ of $G$. If $m(MK)<1$ then either {\bf C1} or {\bf C2} hold once again, so either (1) or (2) holds and we are done. We may therefore assume that $m(MK)=1$. Since $M$ is finite and normal in $G$, $C_K(M)$ is a finite index normal subgroup of $K$, and $NC_K(M)$ has finite index in $G$, hence $m(NC_K(M))=1$. Thus, $m(MC_K(M))=m(NC_K(M)\cap MK)= 1$, so {\bf C3} holds, and hence (3) holds. It remains to prove the implications {\bf Cj}$\Rightarrow$(j) for $j=1,2,3$.

Assume {\bf C1}. Since $\varphi (m)$ is $K$-invariant, if $\varphi (m)$ is not purely atomic then, after renormalizing $\varphi (m)$ on its atomless part if necessary, we see that (1) holds, and if $\varphi (m)$ is purely atomic then the action $K\curvearrowright X$ has infinitely many finite orbits, so (1) holds nonetheless.

Assume {\bf C2}. There must be a finite $K$-invariant subset $X_0\subseteq X$ such that $\varphi (m) (X_0)=1$ (in particular, the action $K\curvearrowright X$ has a finite orbit). Since $\varphi (m)(X_0)=1$, we have that $(\pi _*m_0)(S_0)=1$, where
\[
S_0\defeq  \Big\{ z\in N\setminus \{ 1_N \} \, : \, \frac{|X_0\cap \mathrm{supp}(z)|}{|\mathrm{supp}(z)|}> \frac{|X_0|}{|X_0|+1} \Big\} .
\]
Since the number $\frac{|X_0\cap \mathrm{supp}(z)|}{|\mathrm{supp}(z)|}$ belongs to the set $\{ i/|\mathrm{supp}(z)| \, : \, 0\leq i \leq |X_0| \}$, which (by considering the cases $|\mathrm{supp}(z)|>|X_0|$ and $|\mathrm{supp}(z)|\leq |X_0|$) is seen to be disjoint from the open interval with endpoints $|X_0|/(|X_0|+1)$ and $1$, it follows that any $z\in S_0$ satisfies $\frac{|X_0\cap \mathrm{supp}(z)|}{|\mathrm{supp}(z)|} = 1$, i.e., $\mathrm{supp}(z)\subseteq X_0$. Therefore,  $(\pi _*m_0)(N_0)=1$, where $N_0 \defeq  \{ z\in N \, : \, \mathrm{supp}(z) \subseteq X_0 \}$. This implies that $m(N_0K)=1$. Since $X_0$ is $K$-invariant, $N_0$ is a normal subgroup of $G$, and since $X_0$ is additionally finite, the set $K_0\defeq C_K(N_0)=\bigcap _{x\in X_0}K_x$ has finite index in $K$, and hence $NK_0$ has finite index in $G$. Thus, $m(NK_0)=1$ and therefore $m(N_0K_0)=m(N_0K\cap NK_0)=1$. Since $N_0$ and $K_0$ commute, the restriction of the map $\pi$ to $N_0K_0$ is equivariant for the conjugation action of $N_0$ on $N_0K_0$ and $N_0$ respectively, and hence $\pi _*m$ is invariant under conjugation by $N_0$. Since $N_0C_G(N_0)$ has finite index in $G$, the group $M_0$, of all elements of $N_0$ with finite $N_0$-conjugacy class, is finite is contained in $M$. Therefore, $\pi _*m(M_0)\leq m(MK)<1$, and hence $\pi _*m$ must not be purely atomic. This shows that $N_0$ is inner amenable, and since $X_0$ is finite and $N_0\cong \bigoplus _{X_0}H$, finitely many applications of Proposition \ref{prop:normvsquot} show that $H$ is inner amenable and hence (2) holds.

Assume {\bf C3}. Since $M\leq N$ is finite and normal in $G$, the set $X_0 \defeq  \bigcup _{z\in M}\mathrm{supp}(z)$ is a finite (possibly empty, if $M$ is trivial) $K$-invariant subset of $X$. Put $K_0\defeq \bigcap _{x\in X_0}K_x$ (and $K_0\defeq K$ if $M$ is trivial). Then $M=\{ z\in K\, : \, \mathrm{supp}(z)\subseteq X_0\text{ and }z(X_0)\subseteq M_H \}$, from which it follows that $K_0=C_K(M)$.

Given $z\in N$, we claim that $C_{G/M}(z)\cap z^{-1}MK_0z\cap MK_0 = MC_{K_0}(z)$. The containment $\supseteq$ is clear since $M$ is normal in $G$, so we must show the inclusion $\subseteq$. We can write $z=z_1z_0=z_0z_1$ where $\mathrm{supp}(z_0)\subseteq X_0$ and $\mathrm{supp}(z_1)\subseteq X\setminus X_0$. Let $g\in C_{G/M}(z)\cap z^{-1}MK_0z\cap MK_0$. Then $g=xk$ for some $x\in M$ and $k\in K_0$, and $zxkz^{-1}k^{-1}\in M$. Since $z_1$ commutes with $M$ and $z_0$ commutes with $K_0$ we have $zxkz^{-1}k^{-1} = z_0xz_0^{-1}z_1kz_1^{-1}k^{-1}$, and since $z_0xz_0^{-1}\in M$ this implies that $z_1kz_1^{-1}k^{-1}\in M$. But then $\mathrm{supp}(z_1kz_1^{-1}k^{-1})\subseteq X_0$, and also $\mathrm{supp}(z_1kz_1^{-1}k^{-1})\subseteq X\setminus X_0$, since $\mathrm{supp}(z_1)\subseteq X\setminus X_0$ and $X_0$ is $K_0$-invariant. Therefore, $z_1kz_1^{-1}k^{-1} = 1$ and so $k\in C_{K_0}(z_1)\cap C_{K_0}(z_0)\leq C_{K_0}(z)$.

It now follows that $m(MC_{K_0}(z))= 1$ for all $z\in N$. Define now the mean $m_K$ on $K$ by $m_K(A)\defeq  m(MA)$ for $A\subseteq K$. This is $K$-conjugation invariant since $M$ is normal in $G$. In addition, $m_K(K_x)=1$ for all $x\in X$ since $m_K(C_K(z))=m(MC_K(z))=1$ for all $z\in N$. This shows that (3) holds.
\end{proof}

\noindent {\bf Example:} Consider a Baumslag-Solitar group $G=\mathrm{BS}(p,q)=\langle t,a | ta^pt^{-1}=a^q \rangle$ with $1<p<q$. Let $A$ be the cyclic subgroup generated by $a$. Let $G$ act on $X= G/A$ by translation and consider the wreath product $W\defeq  H\wr _X G$, where $H$ is a non-trivial group. Let $S= \bigoplus _X H$, and identify $S$ and $G$ with their images in $W$. By \cite{MR2226017}, we can find a $G$-conjugation invariant atomless mean $m$ which satisfies $m(A_0)=1$ for every finite index subgroup $A_0$ of $A$. We claim that $m$ is in fact conjugation invariant for all of $W$. It suffices to show it is invariant under conjugation by elements of $S$. Given $s\in S$ let $Q_s\subseteq X$ be the support of $s$. Then the pointwise stabilizer $A_{Q_s}=\bigcap _{x\in Q_s}A_x$ of $Q_s$ in $A$ is a finite index subgroup of $A$, so satisfies $m(A_{Q_s}) = 1$. Since $A_{Q_s}\subseteq C_G(s)$, we have $m(C_G(s))=1$, and hence $m$ is invariant under conjugation by $s$.

\section{Inner amenability and \cat cube complexes}
We now study the structure of inner amenable groups acting on CAT(0) cube complexes.

\subsection{Generalities}
A \cat cube complex is, roughly speaking, a non-positively curved complex built from cubes - i.e.\ subspaces isometric to Euclidean cubes $[0,1]^n$ for some $n\in\Nb$ - glued together via isometries along a face. For an introduction to \cat cube complexes, we refer to \cite{MR3329724}; we here rely primarily on \cite{MR2827012}. We emphasize that the metric statements in this section are always for the \cat metric.

A \cat cube complex $X$ has finite dimension if there some $d$ such that every cube of $X$ has dimension at most $d$. The dimension of $X$ is the least such $d$.

The \textbf{link} of a vertex $x$ in a \cat cube complex $X$ is the simplicial complex whose vertex set is the set of edges of $X$ incident to $x$. A set of $n+1$ vertices of this complex corresponds to a $n$-simplex if and only if the corresponding edges lie in a common cube.

A group $G$ acting by simplicial automorphisms on an irreducible \cat cube complex $X$ acts \textbf{elementarily} if there is an invariant  subset $S\subset \partial X$ with at most 2 elements. Let us emphasize that this notion of elementarity is not the usual one used for groups acting on general \cat spaces (where elementarity refers to the existence either of an invariant Euclidean subspace or a fixed point at infinity). This notion coincides with elementarity for groups acting on Gromov hyperbolic spaces without bounded orbits. The action is \textbf{essential} if no ${G}$-orbit remains in a bounded neighborhood of some half-space of $X$.

For a half-space $\mathfrak{h}$, we denote by $\hat{\mathfrak{h}}$ the corresponding hyperplane and by $\mathfrak{h}^*$ the other half-space corresponding to $\hat{\mathfrak{h}}$. Sometimes, for a given hyperplane $\widehat{\mathfrak{h}}$, we let $\mathfrak{h}^\pm$ denote its two corresponding half-spaces. Two half-spaces $\mathfrak{h}$ and $\mathfrak{k}$ are \textbf{facing} if $\mathfrak{h}\cap\mathfrak{k}\neq\emptyset$ and they are not nested. A \textbf{facing triple} is a triple of pairwise facing halfspaces.


For a group ${G}$ acting by isometries on a \cat space $(X,d)$, the translation length of an element ${g}\in{G}$ is
\[
|{g}|\defeq \inf_{x\in X}d({g} (x),x).
\]
An element ${g}\in {G}$ is called \textbf{elliptic} if there is a fixed point. That is to say, $|{g}|=0$, and the infimum value is achieved. We say that ${g}$ is \textbf{hyperbolic} if $|{g}|>0$ and the infimum is achieved for some $x\in X$. Elements which are neither elliptic nor hyperbolic are called \textbf{parabolic}. For a finite dimensional \cat cube complex, it can be shown that there are no parabolic isometries \cite{MR1646316}.

These three classes of elements are disjoint and invariant under conjugation, because the action is by isometries. We denote these three subsets of ${G}$ by $\Ell({G})$, $\Hyp({G})$, and $\Par({G})$, respectively.

A hyperbolic isometry ${g}$ has \textbf{axes} which are isometrically embedded real lines on which ${g}$ acts as a translation of length $|{g}|$.  A hyperbolic isometry is \textbf{contracting} if there is an axis $L$ and some $R>0$ with the property that each ball that does not intersect $L$ has a projection to $L$ with diameter at most $R$. A contracting isometry has exactly 2 fixed points at infinity; an attractive one and a repulsive one which are the end points of any of its axes.

A \cat cube complex $X$ is said to be \textbf{pseudo-Euclidean} it has an $\Aut(X)$-invariant isometricaly embedded Euclidean subspace. If the space has dimension 1, the complex is said to be $\Rb$-\textbf{like}.

We will need two results from the work \cite{MR2558883} of Caprace and Lytchack. They introduced the telescopic dimension of \cat spaces. A finite dimensional \cat cube complex has finite telescopic dimension.

\begin{prop}[{Caprace--Lytchack, \cite[Theorem 1.1]{MR2558883}}]\label{prop:cp}
Let $X$ be a  complete \cat space of finite telescopic dimension and $(X_i)_{i\in\Nb}$ be a nested sequence of closed convex subspaces of $X$. If $\bigcap X_i=\emptyset$, then $\bigcap\bo X_i$ is nonempty and has radius at most $\pi/2$. In particular, the center of all the circumcenters of $\bigcap\bo X_i$ is unique.\end{prop}

\begin{cor}[{Caprace--Lytchack, \cite[Corollary 1.5]{MR2558883}}] Let $X$ be a complete \cat space of finite telescopic dimension. For a parabolic element $g\in\Isom(X)$, the set of $g$-fixed points in $\partial X$ is nonempty and has a canonical point $\xi_0(g)$ which is the center of the circumcenters of $g$-fxed points in $\partial X$.\end{cor}

This implies that, for a parabolic element $g$, one has $\xi_0(hgh^{-1})=h\xi_0(g)$ for any $h\in \Isom(X)$; that is, the point $\xi_0(g)$ is equivariant in $g$. For a hyperbolic element $g$, we denote by $\xi_\pm(g)$ the attracting and repulsing fixed points of $g$. These are the boundary points such that for any $x\in X$, $g^{\pm n}(x)\to\xi_{\pm}(g)$. Again these points are equivariant with respect to $g$: for any $h\in\Isom(X)$, $\xi_\pm(hgh^{-1})=h\xi_\pm(g)$.

\subsection{Main theorem}


We first start with a general statement for groups equipped with a conjugation invariant mean acting on \cat spaces.

\begin{prop}\label{prop:ell_cat}
Let $X$ be complete \cat space of finite telescopic dimension without Euclidean de Rham factor and let ${G}$ be a group acting minimally on $X$ without fixed points at infinity. If $m$ is a conjugation invariant  mean on ${G}$, then $m(\Ell({G}))=1$.\end{prop}


\begin{proof} Assume that $m(\Par({G}))>0$ (respectively $m(\Hyp({G}))>0$). By considering $m|_{\Par({G})}$ and  renormalizing it, we may assume that $m(\Par({G}))=1$ (respectively $m(\Hyp({G}))=1$). Pushing forward $m$ via $\xi_0$ (respectively $\xi_+$), we get a ${G}$-invariant mean on $\bo X$. One can think of this mean as a positive linear functional $\mu\colon\ell^\infty(\bo X)\to\Rb$ that is ${G}$-invariant. Let us fix $x_0\in X$. For $\xi\in\bo X$, let us denote by $x\mapsto\beta_\xi(x)$ the Busemann function associated to $\xi$ that vanishes at $x_0$. For $x\in X$ and any $\xi\in\partial X$, $|\beta_\xi(x)|\leq d(x,x_0)$. Hence, $\xi\mapsto \beta_\xi(x)$ is a bounded function on $\partial X$. In particular, one can ``integrate'' this function with respect to $\mu$ to obtain a real number that we denote by  $\int_{\partial X}\beta_\xi(x)d\mu(\xi)$. Let us denote by $f$ the function $x\mapsto \int_{\partial X}\beta_\xi(x)d\mu(\xi) $. For any $\xi\in\partial X$, the function $x\mapsto \beta_\xi(x)$ is  1-Lipschitz and convex, and furthermore, $\mu$ is linear and positive with $\mu(\mathbf{1})=1$. It now follows that $f$ is a convex  1-Lipschitz function on $X$ that is quasi-invariant. Moreover it lies in the closed convex hull $\mathcal{C}$ (which is compact) of Busemann functions vanishing at $x_0$. Actually, since $\Rb^X$ is endowed with the pointwise convergence topology, to prove that $f\in\mathcal{C}$, by Hahn-Banach theorem, it suffices to check that for any $x\in X$, $$\inf_{\varphi\in\mathcal{C}}\varphi(x)\leq f(x)\leq\sup_{\varphi\in\mathcal{C}}\varphi(x).$$ This follows from the definition of $f$ and the positivity of $m$.
If $f$ has no minimum then the intersection of its sublevel sets yields a ${G}$-invariant point at infinity. So $f$ has a minimum $m$. The set $\{x\in X\mid f(x)=m\}$ is closed convex and ${G}$-invariant. So by minimality, this subset coincides with $X$, that is $f$ is actually constant.
By \cite[Proposition 4.8 and Lemma 4.10]{MR2558883}, $X$ has a non-trivial Euclidean de Rham factor which is a contradiction.\end{proof}


\begin{cor}\label{ell}
Let $X$ be an irreducible \cat cube complex of finite dimension and let ${G}$ be a group acting essentially and non-elementarily by automorphims on $X$. If $m$ is a conjugation invariant  mean on ${G}$, then $m(\Ell({G}))=1$.\end{cor}

\begin{proof}
Since there is no fixed point at infinity, there is a minimal ${G}$-invariant \cat subspace $Y\subset X$ \cite[Proposition 1.8(ii)]{MR2558883}.	

Thanks to Proposition \ref{prop:ell_cat}, it suffices to show that $Y$ has trivial Euclidean de Rham factor. So, let us assume that $Y$ has a splitting $Y=Y'\times E$ where $E$ is the Euclidean de Rham factor of $Y$ and $\dim(E)\geq1$. By \cite[Theorem 6.3]{MR2827012}, ${G}$ has a contracting isometry (with an axis in $Y$) and thus $\dim(E)\leq1$. So $\dim(E)=1$. The boundary $\partial E$ gives a ${G}$-invariant pair of points in $\partial Y\subset\partial X$ and thus a contradiction with non-elementarity. \end{proof}

\begin{prop}\label{ccc} Let ${G}$ be a group acting essentially and non-elementarily on a finite dimensional irreducible \cat cube complex $X$. If $m$ is a conjugation invariant mean on ${G}$, then
\[
m(\{{g}\in{G} \mid \Fix({g})\cap\mathfrak{h}\})=1
\]
for every half-space $\mathfrak{h}$.
\end{prop}

%
%


\begin{proof}





We consider two collections of hyperplanes of $X$:
\[
\mathcal{W}_1\defeq \Big\{\hat{\mathfrak{h}}\mid \ m\big(\{{g}\mid \Fix({g})\cap\mathfrak{h}\neq\emptyset\text{ and }\Fix({g})\cap\mathfrak{h}^*\neq\emptyset\}\big)=1\Big\},
\]
and
\[
\mathcal{W}_2\defeq \Big\{\hat{\mathfrak{h}} \mid \ m\big(\{{g}\mid \Fix({g})\subset\mathfrak{h}\}\big)>0\text{ or }m\big(\{{g}\mid \Fix({g})\subset\mathfrak{h}^*\}\big)>0\Big\}.
\]
By Proposition \ref{ell}, we know that $m(\Ell({G}))=1$, hence $\mathcal{W}=\mathcal{W}_1\sqcup\mathcal{W}_2$ where $\mathcal{W}$ is the set of all hyperplanes of $X$. We aim to show that $\mathcal{W}_2=\emptyset$.

Let us argue that $\mathcal{W}_1$ and $\mathcal{W}_2$ are transverse. That is, for any $\widehat{\mathfrak{h}}_i\in\mathcal{W}_i$, $\widehat{\mathfrak{h}}_1\cap\widehat{\mathfrak{h}}_2\neq\emptyset$. Suppose toward a contradiction that there are $\hat{\mathfrak{h}}_1\in\mathcal{W}_1$ and $\hat{\mathfrak{h}}_2\in\mathcal{W}_2$ that are not transverse. We  may assume that $\mathfrak{h}_1$ and $\mathfrak{h}_2$ are facing. Since $\mathfrak{h}_2^*\subset\mathfrak{h}_1$, $m(\{{g}\mid \Fix({g})\subset\mathfrak{h}_2^*\})=0$. On the other hand, the double skewer lemma, \cite[Section 1.2]{MR2827012}, ensures that there is $\delta \in{G}$ such that $\delta \mathfrak{h}_1\subset\mathfrak{h}_2^*$. Hence, $m(\{{g}\mid \Fix({g})\subset\mathfrak{h}_2\})=0$ which is a contradiction. We conclude that $\mathcal{W}_1$ and $\mathcal{W}_2$ are transverse.

By \cite[Lemma 2.5]{MR2827012}, $X$ splits as a product of CAT(0) cube complexes $X_1\times X_2$ (with possibly a trivial factor) where $\mathcal{W}_1$ and $\mathcal{W}_2$ are respectively the hyperplane systems of $X_1$ and $X_2$. By irreducibility, we know that one of this factor is trivial. If $\mathcal{W}_2\neq\emptyset$ then $X=X_2$.
In this case, define
\[
\mathcal{W}^\prime_2\defeq \left\{\hat{\mathfrak{h}}\in\mathcal{W}_2\mid\ m(\{{g}\mid \Fix({g})\subset\mathfrak{h}\})\neq m(\{{g}\mid \Fix({g})\subset\mathfrak{h}^*\})\right\}.
\]
For $\hat{\mathfrak{h}}\in\mathcal{W}^\prime_2$, we may choose the half-space $\mathfrak{h}$ such that
\[
m(\{{g}\mid \Fix({g})\subset\mathfrak{h}\})> m(\{{g}\mid \Fix({g})\subset\mathfrak{h}^*\}).
\]
Take $\hat{\mathfrak{h}}$ and $\hat{\mathfrak{k}}$ in $\mathcal{W}^\prime_2$ and suppose that $\mathfrak{h}\cap\mathfrak{k}=\emptyset$. It is then the case that $\mathfrak{k}\subset\mathfrak{h}^*$. Hence,
\[
m\left(\{{g}\mid \Fix({g})\subset\mathfrak{h}\}\right)> m(\{{g}\mid \Fix({g})\subset\mathfrak{h}^*\})\geq m(\{{g}\mid \Fix({g})\subset\mathfrak{k}\}).
\]
On the other hand, $\mf{h}\subset\mf{k}^*$, so the same computation with the roles reversed gives that  $m(\{{g}\mid \Fix({g})\subset\mathfrak{h}\})<m(\{{g}\mid \Fix({g})\subset\mathfrak{k}\})$, which is a contradiction.

It now follows, thanks to the Helly-type property for \cat cube complexes, that $\bigcap_{\hat{\mf{h}}\in \mc{W}^{\prime}_2}\mathfrak{h}$ is an intersection of nested closed convex sets. The intersection $\bigcap_{\hat{\mf{h}}\in \mc{W}^{\prime}_2}\mathfrak{h}$  is  furthermore ${G}$-invariant, and since the action of ${G}$ on $X$ is essential, this intersection is empty. Thanks to Proposition \ref{prop:cp}, there is a global fixed point at infinity for the action of ${G}$ contradicting that the action is non-elementary. We conclude that $\mathcal{W}^\prime_2=\emptyset$.

For any Hyperplane $\hat{\mathfrak{h}}\in\mathcal{W}_2$, it is thus the case that
\[
m(\{{g}\mid \Fix({g})\subset\mathfrak{h}\})= m(\{{g}\mid \Fix({g})\subset\mathfrak{h}^*\}),
\]
and this value is non-zero. Since $X$ is irreducible, $\Aut(X)$ has no invariant Euclidean subspace otherwise  it would be $\Rb$-like (\cite[Lemma 7.1]{MR2827012}) and thus, there would be an invariant pair at infinity corresponding to the ends of the invariant line. Let us now consider a facing triple $\mathfrak{h}_1,\mathfrak{h}_2,\mathfrak{h}_3\in\mathcal{W}_2$, which exists thanks to \cite[Theorem 7.2]{MR2827012}. That the triple is facing ensures that $\mf{h}_j^{*}\subseteq \mf{h}_k$ for $k\neq j$. In particular, $\mf{h}_i^*\cap \mf{h}_j^*=\emptyset$ for $i\neq j$. Setting $\alpha_i\defeq m(\{{g}\mid \Fix({g})\subset\mathfrak{h}^*_i\})$, we see that $\alpha_1\geq\alpha_2+\alpha_3$ and $\alpha_2\geq\alpha_1+\alpha_3$. Hence, $\alpha_3=0$ which is a contradiction.
\end{proof}

\begin{lem}\label{lem:bounded}Let ${G}$ be a group acting essentially and non-elementary on an irreducible \cat cube complex $X$. Let $m$ be any conjugation invariant mean on ${G}$. Then there is some $x_0\in X$ and some $C>0$ such that
\[
m(\{{g}\in{G}\mid\Fix({g})\cap B(x_0,C)\neq\emptyset\})=1.
\]
Moreover the set $X_C$, of all such points, is convex and ${G}$-invariant.
\end{lem}

The lemma follows essentially from Proposition \ref{ccc} and the ideas that gives the existence of contracting isometries in \cite{MR2827012}. We urge the reader to have a copy of \cite{MR2827012} at hand to follow the proof.
\begin{proof}
By the first paragraph of the proof of \cite[Theorem 6.3]{MR2827012}, there is ${g} _0\in {G}$ that skewers some hyperplan $\widehat{\mathfrak{h}}$, with ${g} _0\mathfrak{h}^+\subset\mathfrak{h}^+$, and such that $\widehat{\mathfrak{h}}$ and ${g} _0\widehat{\mathfrak{h}}$ are strongly separated. Let $x_0$ be the intersection of some axis of ${g} _0$ and $\widehat{\mathfrak{h}}$. By the proof of \cite[Lemma 6.1]{MR2827012}, there is $C>0$ such that for any $a\in{g} _0^{-1}\mathfrak{h}^-$ and $b \in{g} _0\mathfrak{h}^+$, the geodesic $[a,b]$ meets $B(x_0,C)$. By Proposition \ref{ccc}, there is a measure 1 set of elements of ${G}$ such that any ${g}$ in this set, $\Fix({g})\cap {g}^{-1}_0\mathfrak{h}^-\neq\emptyset$ and $\Fix({g})\cap {g} _0\mathfrak{h}^+\neq\emptyset$. Thus for any ${g}$ in this set, $\Fix({g})\cap B(x_0,C)\neq\emptyset$.

Since the mean $m$ is conjugation invariant, it follows directly from the definition that $X_C$ is ${G}$-invariant. Now let $x,y\in X_C$ and let $c\colon[0,1]\to X$ be a parametrization of $[x,y]$ proportional to arc-length. There is a measure 1 subset ${G} _0\subset {G}$ such that for all ${g}\in{G}_0$, $\Fix({g})\cap B(x,C)\neq\emptyset$ and $\Fix({g})\cap B(y,C)\neq\emptyset$. Fix ${g}\in{G}_0$, $x'\in \Fix({g})\cap B(x,C)$ and $y'\in\Fix({g})\cap B(y,C)$. Let $c’$ be the similar parametrization of $[x',y']$. By convexity of the metric \cite[II.2.2]{MR1744486}, $d(c(t),c’(t))\leq C$ for any $t\in[0,1]$. Since $c'(t)$ is fixed by ${g}$ as well, the result follows.
\end{proof}

A mean on a group ${G}$ is said to be an \textbf{idempotent} if $m\ast m=m$.

\begin{lem}\label{prop:almost-fix}
Let ${G}$ be a group acting minimally by isometries on some complete \cat space $X$. Let $m$ be a conjugation invariant idempotent mean on ${G}$. Let $C>0$ and suppose $m(\{{g}\in {G}\mid \Fix({g})\cap B(x,C)\neq\emptyset\})=1$ for every $x\in X$.

Then $m(\{{g}\in{G}\mid d(x,{g} x)<\varepsilon\})=1$ for every $x\in X$ and every $\varepsilon>0$.
\end{lem}

\begin{proof}Let us fix $x_0\in X$.
 For ${g}\in{G}$ such that $\Fix({g})\neq\emptyset$ and $x\in X$,
	$$d(x,{g} x_0)\leq d(x,\Fix({g}))+d(\Fix({g}),x_0)$$
	and thus the following formula
	$$\varphi(x)=\int_{G} d(x,{g} x_0)^2dm({g})$$
	defines a  function $\varphi\colon X\to \Rb^+$. The continuity follows from the inequality
	\begin{align*} \left| d(x,{g} x_0)^2-d(y,{g} x_0)^2\right|&= \left| d(x,{g} x_0)-d(y,{g} x_0)\right|\cdot\left(d(x,{g} x_0)+d(y,{g} x_0)\right) \\
	&\leq d(x,y)\left( 2d(x,{g} x_0)+d(x,y)\right)
	\end{align*}
	which integrates into
	\begin{equation*}
		\left|\varphi(x)-\varphi(y)\right|\leq d(x,y)\left(2\varphi(x)+d(x,y)\right).
	\end{equation*}
It follows from linearity and the \cat inequality that for any $x,y\in X$ and $c$ their midpoint that
	\begin{equation}\label{eq:strctconv}\varphi(c)\leq \frac{1}{2}\left(\varphi(x)+\varphi(y)\right)-\frac{1}{4}d(x,y)^2.\end{equation}
We claim that $\varphi$ has a unique minimum $z$ and for any $z'\in X$, $$d(z,z')^2\leq 2(\varphi(z')-\varphi(z)).$$ Only the existence of this minimum requires an argument, the inequality and the uniqueness follow directly from Equation \eqref{eq:strctconv}. So let us prove the existence. For ${g}$ in a set of measure 1, $d(x_0,\Fix({g}))\leq C$ and thus for any ${g}$ in this set, $d(x_0,{g} x_0)\leq 2C$. In particular, $\varphi(x_0)\leq 4C^2$. Now, for a point $y\in X\setminus B(x_0,4C)$, the reverse triangle inequality, implies that $d(y,{g} x_0)\geq 2C$ and thus $\varphi(y)\geq 4C^2$. Since any continuous convex function on a bounded complete \cat space has a minimum (this follows from \cite[Theorem 14]{MR2219304}), $\varphi$ has a minimum which lies in $B(x_0,4C)$.

By idempotence of $m$ we have,
\begin{align*}\varphi(z)&= \int_{G} d(z,{g} x_0)^2dm\ast m({g})\\
&=\int_{h\in{G}}\int_{k\in {G}}d(h^{-1}z,kx_0)^2dm(k)dm(h)	\\
&=\int_{h\in{G}}\varphi(h^{-1}z)dm(h).
\end{align*}

In particular, for any $\varepsilon>0$ there is a measure 1 set of elements $h\in{G}$ such that $\varphi(h^{-1}z)<\varphi(z)+\varepsilon$ and thus $d(z,h^{-1}z)^2\leq2\varepsilon$. This gives that $\{x\in X\mid \forall \varepsilon >0, m(\{{g}\mid d(x,{g} x)<\varepsilon\})=1\}$ is nonempty. Since this set is closed convex and ${G}$-invariant, it is $X$ itself.
\end{proof}

\begin{thm}\label{thm:fix_ccc}
Let ${G}$ be a group acting essentially and non-elementary on an irreducible \cat cube complex $X$.
There is a nonempty closed convex subspace $X_0$ such that for any conjugation invariant mean $m$ on ${G}$ and any $x\in X_0$, $m({G}_x)=1$.
\end{thm}

\begin{proof}
Since the action is non-elementary, there are minimal invariant closed convex subspaces. Moreover the union of all such minimal subspaces split as a product $X_0\times C$ \cite[Theorem 4.3 (B.ii)]{MR2574740} where $X_0$ is one of these minimal subspaces and the action is diagonal, being trivial on $C$. So, all minimal subspaces are equivariantly isometric. Let us fix a minimal closed subspace $X_0$.

 Assume first that $n$ is a conjugation invariant mean on ${G}$ that is additionally idempotent. Then we can apply Lemma \ref{prop:almost-fix} to some minimal closed convex ${G}$-invariant subspace $X_1$ of $\overline{X_C}$ given by Lemma \ref{lem:bounded}. Since the orbit of any point in $X$ is discrete under the action of $\Aut(X)$, we conclude that $n({G}_x)=1$ for every $x\in X_1$. Since $X_1$ and $X_0$ are equivariantly isometric, for every  $x\in X_0$, $n({G}_x)=1$.


Now let $m$ be a conjugation invariant mean on ${G}$. Since the map $n\mapsto n\ast m$ is affine, continuous, and the set of conjugation invariant means is a convex compact subspace of a locally convex topological vector space, the Markov-Kakutani fixed point theorem gives the existence of an \textbf{$m$-stationary} conjugation invariant mean, i.e., a mean $n$ satisfying the equation $n\ast m=n$. The set of all such $m$-stationary conjugation invariant means is a compact left topological semigroup for convolution, hence Ellis's Lemma \cite[Lemma 1]{MR0101283} gives the existence of an $m$-stationary conjugation invariant mean $n$ which is furthermore idempotent. By the paragraph above, we have $n({G} _x)=1$ for every $x\in X_0$. Therefore, by Lemma \ref{lem:stationary}, since $n$ is $m$-stationary, we conclude that $m({G}_x)=1$ for every $x\in X_0$.
	\end{proof}

\subsection{A few applications}


\begin{cor}\label{cor:non_irr_CC}Let ${G}$ be a group acting on a finite dimensional \cat cube complex $X$ with no fixed points and no finite orbit in $\partial X$. Then there exists a finite index subgroup ${G}_0$ and a closed convex ${G}_0$-invariant subspace $X_0$ such that for any conjugation invariant mean $m$ on ${G}_0$ and for any point $x\in X_0$, $m(({G}_0)_x)=1$.\end{cor}

\begin{proof} Since ${G}$ has no fixed point in $X$, nor any fixed point at infinity, there is a nonempty invariant subcomplex called the essential core, $Y\subset X$ on which ${G}$ acts essentially \cite[Proposition 3.5]{MR2827012}. This complex $Y$ has a canonical splitting $Y=Y_1\times\dots\times Y_n$ which is preserved by $\Aut(Y)$ \cite[Proposition 2.6]{MR2827012}. Let us gather the pseudo-Euclidean factors as the $k$ first ones. That is $Y_i$ is pseudo-Euclidean if and only if $i\leq k$. In particular, if we denote by $Y_{\textrm{euc}}$ the product $Y_1\times\dots\times Y_k$, then $Y_{\textrm{euc}}$ is pseudo-Euclidean and the action $\Aut(Y_{\textrm{euc}})\acts Y_{\textrm{euc}}$ is essential. Let us also denote by $Y_\textrm{non-euc}$ the product of the remaining factors. Observe that the splitting $Y=Y_{\textrm{euc}}\times Y_\textrm{non-euc}$ is $\Aut(Y)$-invariant and this gives an action of $\Aut(Y)$ on $Y_{\textrm{euc}}$. By \cite[Lemma 7.1]{MR2827012}, each factor $Y_i$, for $i\leq k$, is $\Rb$-like. In particular, each factor has a $\Aut(Y_i)$-invariant line which gives an $\Aut(Y_i)$-invariant pair of points in $\partial Y_i$. So this gives a finite orbit in $\partial Y_{\textrm{euc}}$ for $\Aut(Y)$ and thus for ${G}$ as well. This is a contradiction and this implies that $Y$ has no pseudo-Euclidean factor.

Let ${G}_0$ be the finite index subgroup of ${G}$ that preserves each factor of the decomposition  $Y=Y_1\times\dots\times Y_n$. For any $i\leq k$, the action ${G}_0\acts Y_i$ is essential and non-elementary (otherwise there would be a finite ${G}$-orbit in $\partial Y$). By Theorem~\ref{thm:fix_ccc}, there is a closed convex subspace $X_i\subset Y_i$ such that the stabilizer of any point in $X_i$ has measure 1 for any conjugation invariant mean $m$. Let us denote by $X_0=X_1\times\dots\times X_n\subset Y$. By intersecting finitely measure 1 sets, it follows that the stabilizer of any point in $_0$ has measure 1.
\end{proof}


\begin{cor}\label{cor:CC_non-inner-amen}
Let  ${G}$ be a group acting on some finite dimensional \cat cube complex properly and without finite orbit at infinity. Then ${G}$ is not inner amenable.
\end{cor}

\begin{proof}\label{cor:proper} Let us assume toward a contradiction that ${G}$ is inner amenable.  We continue with the same notations as in Corollary \ref{cor:non_irr_CC}. Say that $m$ is a conjugation invariant atomless mean on ${G}_0$ (which exists since it has finite index in ${G}$). By Corollary \ref{cor:non_irr_CC}, the stabilizer of any point in $X_0$ has measure 1. However, that the action is proper ensures the stabilizer of any vertex is finite. Therefore  the mean $m$ has atoms, which is a contradiction.
\end{proof}

\begin{rmk}\label{rmk:DGO} Corollary~\ref{cor:proper} can also be deduced from previous results. From \cite[Theorem 1.3]{sisto2016contracting} (see also \cite[2.23]{MR3589159}), one can deduce the existence of non-degenerate hyperbolically embedded subgroup of ${G}$ and thus the group ${G}$ is not inner amenable \cite[2.35]{MR3589159}.
\end{rmk}

\begin{rmk}\label{rmk:BIP}Under the hypotheses of Corollary~\ref{cor:proper}, we can also show that ${G}$ has a natural proper $1$-cocycle into a non-amenable representation and thus ${G}$ is properly proximal in the sense of \cite{BIP18} and hence not inner amenable. There is a well-known natural $1$-cocycle for the quasi-regular representation of ${G}$ on $\ell^2(\mathcal{H})$, associated to the action of ${G}$ on the set of halfspaces $\mathcal{H}$, and this cocycle is proper if the action of ${G}$ on $X$ is proper \cite[\S1.2.7]{MR1852148}. It remains to show that this representation is non-amenable, which (since we are dealing with a quasi-regular representation) is equivalent to showing that there is no ${G}$-invariant mean on $\mathcal{H}$.

Assume for the sake of contradiction that there is such a ${G}$-invariant mean $m$ on $\mathcal{H}$. We can push it forward via $\mathfrak{h}\mapsto \widehat{\mathfrak{h}}$ to get a ${G}$-invariant mean $m$ on the set of hyperplanes $\mathcal{W}$. Thanks to the same argument as in the proof of Corollary \ref{cor:non_irr_CC}, it suffices to consider the case where $X$ is irreducible and the action is essential. Analogous to the proof of Proposition~\ref{prop:almost-fix}, let us define

$$\mathcal{W}_1=\{\widehat{\mathfrak{h}}\mid m(\{\widehat{\mathfrak{k}} \mid \widehat{\mathfrak{k}} \bot \widehat{\mathfrak{h}}\})=1\},$$
$$\mathcal{W}_2=\{\widehat{\mathfrak{h}}\mid m(\{\widehat{\mathfrak{k}} \mid \widehat{\mathfrak{k}}\subset \mathfrak{h} \} )>0\ \textrm{or}\ m(\{\widehat{\mathfrak{k}}\mid \widehat{\mathfrak{k}}\subset \mathfrak{h}^\ast \} )>0\}.$$

For similar reasons as in Proposition~\ref{prop:almost-fix}, $\mathcal{W}=\mathcal{W}_1\sqcup\mathcal{W}_2$ and these collections are transverse. By irreducibility, one is trivial and arguing with $\mathcal{W}_2$ as in the proof of Proposition~\ref{prop:almost-fix}, we show that $\mathcal{W}=\mathcal{W}_1$. Thus, for any two $\widehat{\mathfrak{h}_1}, \widehat{\mathfrak{h}_2}$, there is a measure 1 set of hyperplanes which simultaneously cross them both. This contradicts the existence of strongly separated hyperplanes which is guaranteed thanks to the irreducibility of the complex \cite[Proposition 5.1]{MR2827012}.
\end{rmk}

\begin{exs}The Higman group is not inner amenable.
 Let us recall that the Higman group $H$  is the group given by the presentation.
\[H=\langle a_0, a_1,a_2,a_3\mid  a_ia_{i+1}a_i^{-1}=a_{i+1}^2\text{ with } i\in\Zb/4\Zb\rangle.\]

The work \cite{MR3619304} exhibits an action of $H$ on an irreducible \cat square complex. From the description of the action, it follows that the convex hull of any orbit meets the interior of some square and the action is essential and non-elementary. Moreover, the action on squares is regular So , by Theorem~\ref{thm:fix_ccc}, for any conjugacy invariant mean on $H$ $m(\{1\})=1$, so $H$ is not inner amenable.
\end{exs}

\subsection{Graph products of groups}
Let $\Gamma$ be a  finite simplicial graph with vertex set $V\Gamma$ and edge set $E\Gamma$. The neighborhood $N(v)$ of $v\in V\Gamma$ is the set $\{w\in V\Gamma, \ w=v\ \textrm{or}\ \{v,w\}\in E\Gamma\}$. A \textbf{clique} in $\Gamma$ is subset $C\subset V\Gamma$ such that the induced graph is complete. By a \textbf{maximal clique}, we mean a clique which is maximal for inclusion. The flag simplicial complex $F\Gamma$ associated to $\Gamma$ is the simplicial complex with 1-skeleton $\Gamma$ and simplices corresponding to cliques.

 Assume that for each $v\in V\Gamma$, a non-trivial group $G_v$ is given. The groups $G_v$ are called the \textbf{vertex groups}.  For any simplex $\sigma$ of $F\Gamma$, that is a clique in $\Gamma$, we define $G_\sigma=\prod_{v\in\sigma}G_v$. In particular, for two simplices $\tau\subset \sigma$, we have the natural inclusion $\psi_{\sigma\tau}\colon G_\tau\to G_\sigma$. The \textbf{graph product} $G_\Gamma$ is the direct limit of the system given by the groups $G_\sigma$ and homomorphisms $\psi_{\sigma\tau}$. It can also been described as the quotient of the free product of all vertex groups by the normal subgroup generated by the commutators $[a,b]$ with $a\in G_v, b\in G_w$ and $\{v,w\}\in E\Gamma$.

 More generally, if $S$ is a subset of $V\Gamma$, we denote by $\Gamma_S$ the subgraph of $\Gamma$ induced by $S$. The subgroup of $G_\Gamma$ generated by $\{G_v\}_{v\in S}$ is denoted by $G_{S}$ and is isomorphic to the graph product $G_{\Gamma_S}$.

 The graph $\Gamma$ is a $\textbf{join}$ if there are two proper subsets $V_1,V_2\subset V\Gamma$, $V\Gamma=V_1\sqcup V_2$ and such that for any $v_1\in V_1$ and  $v_2\in V_2$, $\{v_1,v_2\}\in E\Gamma$. In this case, if $\Gamma_i$ is the  graph induced by $V_i$ then the graph product $G_\Gamma$ splits as the direct product $G_{\Gamma_1}\times G_{\Gamma_2}$. The \textbf{complement graph} $\overline{\Gamma}$ is the graph with same vertex set $V\overline{\Gamma}=V\Gamma$ and $\{v,w\}\in E\overline{\Gamma}$ if and only if $\{v,w\}\notin E\Gamma$. Let $\Gamma_1,\dots,\Gamma_n$ be the subgraphs of $\Gamma$ induced by the vertex sets of the connected components of $\overline{\Gamma}$. By the above remark, the group $G_\Gamma$ splits as a direct product $G_{\Gamma_1}\times\dots\times G_{\Gamma_n}$. This canonical splitting is maximal in the sense that no $\Gamma_i$ is a join.

The goal of this subsection is to prove the following characterization of inner amenability of graph products of groups and to specialize this result to the cases of right-angled Artin groups and right-angled Coxeter groups.

\begin{thm}\label{thm:graph_prod} The graph product $G_\Gamma$ is inner amenable if and only if
\begin{itemize}
\item there is $v\in V\Gamma$ such that $N(v)=V\Gamma$ and $G_v$ is inner amenable or
\item there are $v_1,v_2\in V$ such that $N(v_1)=V\setminus\{v_2\}$, $N(v_2)=V\setminus\{v_1\}$ and $G_{v_1}\simeq G_{v_2}\simeq\Zb/2\Zb$. In particular $G_\Gamma$ splits as direct product with the infinite dihedral group $D_\infty$.
\end{itemize}
\end{thm}

To prove this theorem, we use a nice combinatorial action of $G_\Gamma$ on a \cat cube complex $X_\Gamma$ due to Meier and Davis \cite{MR1389635,MR1709955}. This construction is also described in \cite[Example II.12.30.(2)]{MR1744486}. We refer to these references for an explicit construction. This action has a cubical complex $C\Gamma$ as strict fundamental domain, which is the cubulation of the simplicial complex $F\Gamma$. Let us describe it. The complex $C\Gamma$ is completely determined by its 1-skeleton $C\Gamma^{(1)}$ and thus we only describe this graph (this 1-skeleton is a median graph see \cite[\S6]{MR1748966}).

The vertex set of $C\Gamma^{(1)}$ is the set of cliques of $\Gamma$ together with the empty set (seen as the empty clique). Two cliques $\sigma$ and $\tau$ are joined by an edge if and only if their symmetric difference is a singleton. More generally, two cliques lie in a common cube if and only if their union is a clique. So all maximal cubes of $C\Gamma$ have a set of vertices given by the set of all subsets of a maximal clique. In particular, the link of the vertex $\emptyset$ is $F\Gamma$.

For a point $x\in C\Gamma$, we denote by $\sigma(x)$ the smallest clique (for inclusion) appearing as vertex of the smallest cube containing $x$. The CAT(0) cube complex $X_\Gamma$ (which depends on the vertex groups whereas $C\Gamma$ does not) is obtained as a quotient
$$\left({G_\Gamma\times C\Gamma}\right)/\sim$$
where $(g,x)\sim (h,y)$ if $x=y$ and $g^{-1}h\in G_{\sigma(x)}$.

This quotient is naturally endowed with the cubical structure coming from $C\Gamma$ and $G_\Gamma$ acts by automorphisms via $g\cdot(h,x)=(gh,x)$. For example, the stabilizer of a vertex $\sigma$ is exactly $G_\sigma$ (with the convention that $G_\emptyset=\{1\}$). The vertices of $X_\Gamma$ are in bijection with cosets $gG_\sigma\in G_\Gamma/G_\sigma$, that is the union indexed by the set of all cliques (possibly empty) $$X_\Gamma^{(0)}=\bigsqcup_{\sigma} G_\Gamma/G_\sigma.$$
Observe that $C\Gamma$ embeds in $X_\Gamma$ by the map $x\mapsto (1,x)$. Under this embedding the link of $\emptyset$ in $X_\Gamma$ is the same as in $C\Gamma$, that is $F\Gamma$. This follows from the fact that the stabilizer of $\emptyset$ is $G_\emptyset=\{1\}$.

\begin{exs}To explicit a bit this construction, we illustrate it on some simple examples in Table \ref{table:graphproducts}. For the sake of simplicity, the vertex groups are cyclic but the construction is not restricted to this case.

\begin{table}
    \centering
\begin{adjustbox}{center}
\begin{tabular}{c|c|c|c}
$\Gamma$& \begin{tabular}{c}Vertex groups and\\their graph products\end{tabular}& $C\Gamma$&$X_\Gamma$\\
 \specialrule{2pt}{0pt}{0pt}
	\adjustbox{valign=c}{\begin{tikzpicture}
	\draw (0,0) node {$\bullet$} ;
	 \draw (0,0) node[above]{$v$} ;
	 \end{tikzpicture}}

	&$G_v=G_\Gamma=\Zb/5\Zb$&
	\adjustbox{valign=c}{\begin{tikzpicture}
	\draw (0,0) node {$\bullet$} ;
	\draw (0,0) node[above]{$\emptyset$} ;
	\draw (0,0)--(1,0);
	\draw (1,0) node {$\bullet$} ;
	\draw (1,0) node[above]{$\{v\}$} ;
	 \end{tikzpicture}}&
	\adjustbox{valign=c}{ \begin{tikzpicture}
	 \draw (0,0) node {$\bullet$} ;
	 \foreach \x in {0,...,4}
	 {\draw (\x*72:1) node {$\bullet$} ;
	 \draw (\x*72:1.3) node {$\x$} ;
	 \draw (0,0) -- (\x*72:1);}
	 \draw (-.15,0) node[left]{$\{v\}$};
	 \end{tikzpicture}}\\
	 \hline
	\adjustbox{valign=c}{ \begin{tikzpicture}
	\draw (-.5,0) node {$\bullet$} ;
	 \draw (-.5,0) node[above]{$v$} ;
	 \draw (.5,0) node {$\bullet$} ;
	 \draw (.5,0) node[above]{$v^\prime$} ;
	 \end{tikzpicture}}&\begin{tabular}{c}$G_v=\Zb/3\Zb,\ G_{v^\prime}=\Zb/4\Zb$ \\ \\ $G_\Gamma=\Zb/3\Zb\ast\Zb/4\Zb$\end{tabular}&
	\adjustbox{valign=c}{ \begin{tikzpicture}
	\draw (0,0) node {$\bullet$} ;
	\draw (0,0) node[below]{$\emptyset$} ;
	 \draw (0,0)--(45:1);
	\draw (45:1) node {$\bullet$} ;
	\draw (45:1) node[above]{$\{v\}$} ;
	 \draw (0,0)--(135:1);
	\draw (135:1) node {$\bullet$} ;
	\draw (135:1) node[above]{$\{v^\prime\}$} ;
	\end{tikzpicture}}&
	\adjustbox{valign=c}{ \begin{tikzpicture}[scale=.5]
	\draw (0,0) node {$\bullet$}--(1,0) node {$\bullet$} -- (2,0) node {$\bullet$};
	\draw (0,0) node {$\bullet$}--(-1,0) node {$\bullet$} -- (-2,0) node {$\bullet$};
	\draw (0,0) node {$\bullet$}--(0,1) node {$\bullet$} -- (0,2) node {$\bullet$};
	\draw (0,0) node {$\bullet$}--(0,-1) node {$\bullet$} -- (0,-2) node {$\bullet$};
	\draw (2,0)--++(45:1) node {$\bullet$} node[above right]{\tiny\reflectbox{$\ddots$}};
	\draw (2,0)--++(-45:1) node {$\bullet$} node[below right]{\tiny$\ddots$};
	\draw (0,2)--++(45:1) node {$\bullet$}node[above right]{\tiny\reflectbox{$\ddots$}};
	\draw (0,2)--++(135:1) node {$\bullet$}node[above left]{\tiny$\ddots$};
	\draw (0,-2)--++(-135:1) node {$\bullet$}node[below left]{\tiny\reflectbox{$\ddots$}};
	\draw (0,-2)--++(-45:1) node {$\bullet$} node[below right]{\tiny$\ddots$};
	\draw (-2,0)--++(135:1) node {$\bullet$} node[above left]{\tiny$\ddots$};
	\draw (-2,0)--++(-135:1) node {$\bullet$}node[below left]{\tiny\reflectbox{$\ddots$}};
	\draw (1,0) node[below]{\tiny$\emptyset$};
	\draw (2,0) node[right]{\tiny$\{v\}$};
	\draw (.25,0.05) node[below left]{\tiny$\{v^\prime\}$};
	\end{tikzpicture}}
	\\
	\hline
	\adjustbox{valign=c}{\begin{tikzpicture}
	\draw (-.5,0) node {$\bullet$} ;
	 \draw (-.5,0) node[above]{$v$} ;
	 \draw (.5,0) node {$\bullet$} ;
	 \draw (.5,0) node[above]{$v^\prime$} ;
	 \draw (-.5,0)--(.5,0);
	 \end{tikzpicture}}&\begin{tabular}{c}$G_v=\Zb/2\Zb,\ G_{v^\prime}=\Zb/3\Zb$ \\ \\ $G_\Gamma=\Zb/2\Zb\times\Zb/3\Zb$\end{tabular}&
	\adjustbox{valign=c}{ \begin{tikzpicture}
	 \draw[fill=gray!40] (0,0)--(1,0)--(1,1)--(0,1)--(0,0);
	 \draw (0,0) node {$\bullet$} ;
	 \draw (0,0) node[below left]{$\emptyset$} ;
	 \draw (1,0) node {$\bullet$} ;
	 \draw (1,0) node[below right]{$\{v\}$} ;
	 \draw (0,1) node {$\bullet$} ;
	 \draw (0,1) node[above left]{$\{v^\prime\}$} ;
	 \draw (1,1) node {$\bullet$} ;
	 \draw (1,1) node[above right]{$\{v,v^\prime\}$} ;
	 \end{tikzpicture}}
	 &
	\adjustbox{valign=c}{ \begin{tikzpicture}
	 \draw[fill=gray!40] (0,0)--(1,0)--(1,1)--(0,1)--(0,0);
	 \draw[fill=gray!40] (1,0)--(2,0)--(2,1)--(1,1)--(1,0);
	 \draw[fill=gray!40] (0,1)--(1,1)--(1,2)--(0,2)--(0,1);
	 \draw[fill=gray!40] (1,1)--(2,1)--(2,2)--(1,2)--(1,1);
	\draw[fill=gray!40] (0,1)-- ++(45:1)-- ++(1,0)-- ++(225:1);
	\draw[fill=gray!40] (1,1)-- ++(45:1)-- ++(1,0)-- ++(225:1);
	 \draw (0,0) node {$\bullet$} node[below left]{\scriptsize$\emptyset$} ;
	 \draw (1,0) node {$\bullet$} node[below ]{\scriptsize$\{v\}$} ;
	  \draw (0,1) node {$\bullet$} node[left]{\scriptsize $\{v^\prime\}$} ;
	  \draw (1,1) node {$\bullet$} node[below right]{\scriptsize$\{v,v’\}$} ;
	  \draw (0,2) node {$\bullet$} node[above left] {\scriptsize$(0,1)$};
	  \draw (0,1)++(45:1) node {$\bullet$} ++(.2,0)node[below] {\tiny$(0,2)$};
	   \draw (2,1)+(45:1) node {$\bullet$} node[right] {\tiny$(1,2)$};
	    \draw (2,2) node {$\bullet$} node[above right ] {\tiny$(1,1)$};
	     \draw (2,0) node {$\bullet$} node[below right ] {\tiny$(1,0)$};
	       \draw (1,1)+(45:1) node {$\bullet$};
	 \end{tikzpicture}}\\
	 \hline
	 \adjustbox{valign=c}{\begin{tikzpicture}
	\draw (0,0) node {$\bullet$} ;
	 \draw (0,0) node[below left]{$v$} ;
	 \draw (1,0) node {$\bullet$} ;
	 \draw (1,0) node[below right]{$v^\prime$} ;
	 \draw(60:1) node {$\bullet$} ;
	 \draw (60:1) node[above] {$v^{\prime\prime}$};
	 \draw (0,0)--(1,0);
	 \end{tikzpicture}}&\begin{tabular}{c}$G_v=G_{v^\prime}=G_{v^{\prime\prime}}=\Zb/2\Zb$ \\ \\ $G_\Gamma=\Zb/2\Zb\ast\left(\Zb/2\Zb\times\Zb/2\Zb\right)$\end{tabular}&
	\adjustbox{valign=c}{ \begin{tikzpicture}
	 \draw[fill=gray!40] (0,0)--(1,0)--(1,1)--(0,1)--(0,0);
	 \draw (0,0) node {$\bullet$} ;
	 \draw (0,0) node[ left]{$\emptyset$} ;
	 \draw (1,0) node {$\bullet$} ;
	 \draw (1,0) node[ right]{$\{v\}$} ;
	 \draw (0,1) node {$\bullet$} ;
	 \draw (0,1) node[above left]{$\{v^\prime\}$} ;
	 \draw (1,1) node {$\bullet$} ;
	 \draw (1,1) node[above right]{$\{v,v^\prime\}$} ;
	 \draw (0,0)--(-90:1) node{$\bullet$};
	 \draw (-90:1) node[right]{$\{v^{\prime\prime}\}$};
	 \end{tikzpicture}}
	 &
	\adjustbox{valign=c}{\begin{tikzpicture}[scale=.5]
	  \draw[fill=gray!40] (0,0)--(1,0)--(1,1)--(0,1)--(0,0);
	 \draw[fill=gray!40] (1,0)--(2,0)--(2,1)--(1,1)--(1,0);
	 \draw[fill=gray!40] (0,1)--(1,1)--(1,2)--(0,2)--(0,1);
	 \draw[fill=gray!40] (1,1)--(2,1)--(2,2)--(1,2)--(1,1);
	 \begin{scope}[yshift=3.41cm,xshift=3.41cm]
		\draw[fill=gray!40] (0,0)--(1,0)--(1,1)--(0,1)--(0,0);
	 	\draw[fill=gray!40] (1,0)--(2,0)--(2,1)--(1,1)--(1,0);
	 	\draw[fill=gray!40] (0,1)--(1,1)--(1,2)--(0,2)--(0,1);
	 	\draw[fill=gray!40] (1,1)--(2,1)--(2,2)--(1,2)--(1,1);
	 	\end{scope}
		\begin{scope}[yshift=-3.41cm,xshift=-3.41cm]
		\draw[fill=gray!40] (0,0)--(1,0)--(1,1)--(0,1)--(0,0);
	 	\draw[fill=gray!40] (1,0)--(2,0)--(2,1)--(1,1)--(1,0);
	 	\draw[fill=gray!40] (0,1)--(1,1)--(1,2)--(0,2)--(0,1);
	 	\draw[fill=gray!40] (1,1)--(2,1)--(2,2)--(1,2)--(1,1);
	 	\end{scope}
		\begin{scope}[yshift=+3.41cm,xshift=-3.41cm]
		\draw[fill=gray!40] (0,0)--(1,0)--(1,1)--(0,1)--(0,0);
	 	\draw[fill=gray!40] (1,0)--(2,0)--(2,1)--(1,1)--(1,0);
	 	\draw[fill=gray!40] (0,1)--(1,1)--(1,2)--(0,2)--(0,1);
	 	\draw[fill=gray!40] (1,1)--(2,1)--(2,2)--(1,2)--(1,1);
	 	\end{scope}
		\begin{scope}[yshift=-3.41cm,xshift=+3.41cm]
		\draw[fill=gray!40] (0,0)--(1,0)--(1,1)--(0,1)--(0,0);
	 	\draw[fill=gray!40] (1,0)--(2,0)--(2,1)--(1,1)--(1,0);
	 	\draw[fill=gray!40] (0,1)--(1,1)--(1,2)--(0,2)--(0,1);
	 	\draw[fill=gray!40] (1,1)--(2,1)--(2,2)--(1,2)--(1,1);
	 	\end{scope}
	 \draw (0,0) node {\tiny$\bullet$} node[above left]{\tiny$\emptyset$} --(225:1)node {\tiny$\bullet$} --(225:2)node {\tiny$\bullet$} ;
	\draw (2,0) node {\tiny$\bullet$}  --+(-45:1)node {\tiny$\bullet$} --+(-45:2)node {\tiny$\bullet$} ;
	 \draw (2,2) node {\tiny$\bullet$}  --+(45:1)node {\tiny$\bullet$} --+(45:2)node {\tiny$\bullet$} ;
	 \draw (0,2) node {\tiny$\bullet$}  --+(135:1)node {\tiny$\bullet$} --+(135:2)node {\tiny$\bullet$} ;
	 \draw (4,4)+(45:2) node[above right]{\reflectbox{\tiny$\ddots$}};
	\draw (4,-2)++(-45:2) node[below right]{\tiny$\ddots$};
	\draw (-2,4)++(135:2) node[above left]{\tiny$\ddots$};
	\draw (-2,-2)++(-135:2) node[below left]{\reflectbox{\tiny$\ddots$}};
	 \end{tikzpicture}}\\

\end{tabular}
\end{adjustbox}
\caption{Some examples of graph products and their associated \cat cube complexes.}\label{table:graphproducts}

\end{table}
\end{exs}
 \begin{lem}\label{lem:conv} Let $S\subset V\Gamma$. The \cat cube complex $X_{\Gamma_S}$ embeds as a convex subcomplex of $X_\Gamma$ in a $G_{S}$-equivariant way. \end{lem}

 \begin{proof} By construction, the set of vertices of $X\Gamma_S$ is the union
 $$\bigsqcup_{\sigma\subset S} G_S/G_\sigma $$
 over cliques included in $S$. This can be seen as a subset of $$\bigsqcup_{\sigma\subset V\Gamma} G_\Gamma/G_\sigma$$ and this gives the embedding at the level of vertices. It is clearly $G_S$-equivariant. Moreover, by construction, two vertices of $X_{\Gamma_S}$ lie in a common cube of $X_\Gamma$ if and only if they lie in a common cube of $X_{\Gamma_S}$ and thus the embedding is convex.
 \end{proof}

\begin{lem}\label{lem:nofixedpoint} The action of $G_\Gamma\acts\partial X_\Gamma$ has no fixed point. \end{lem}

\begin{proof}Let $\sigma$ be a maximal clique of the graph $\Gamma$. Assume that $\xi\in\partial X_\Gamma$ is a $G_\Gamma$-fixed point. The geodesic ray $L$ from the vertex $\sigma$ to $\xi$ is pointwise fixed by $G_\sigma$. The maximal cubes having $\sigma$ as vertex are images by some element $g\in G_\sigma$ of the cube $C$ with vertex set $\{\tau\mid \tau\subset\sigma\}$. In particular, if $g\in G_\sigma$ does not lie in any $G_v$ for $v\in\sigma$ then $gC\cap C=\{\sigma\}$. So $G_\sigma$ acts transitively on maximal cubes attached to $\sigma$ and the intersection of all these cubes is reduced to $\sigma$. We have a contradiction because there is some maximal cube attached to $\sigma$ such that the intersection of $L$ and this cube is not reduced to $\{\sigma\}$.
\end{proof}

If $\Gamma$ is a join, the group $G_\Gamma$ splits as direct product and the \cat cube complex $X_\Gamma$ splits as a direct product with factors associated to the factors of $G_\Gamma$. The converse is also true.

\begin{lem}\label{lem:irreducible} If $\Gamma$ is not a join then $X_\Gamma$ is irreducible.
\end{lem}

\begin{proof} In the complex $C\Gamma$, vertices are in bijection with cliques $S\subset V\Gamma$ (possibly empty). Two such vertices are connected by an edge if they differ by one element. In particular, edges with one end $\emptyset$ have some singleton $\{v\}$ for the other end. Since any maximal cube in $C\Gamma$ contains the vertex $\emptyset$, to any hyperplane of $C\Gamma$, one can associate a unique $v\in V\Gamma$, which is the unique $v$ such that this hyperplane is the parallelism class of a unique edge with ends $\emptyset$ and $\{v\}$. We denote by $\mathfrak{h}_v$ this hyperplane (seen as an hyperplane of $C\Gamma$ or $X_\Gamma$).  By construction, for any $v_1,v_2$, $\mathfrak{h}_{v_1}$ crosses $\mathfrak{h}_{v_2}$ if and only if $\{v_1,v_2\}\in E\Gamma$.

Assume that $X_\Gamma$ splits as as product of \cat cube complexes then there is a canonical non-trivial decomposition $X_\Gamma=X_0\times\dots\times X_n$  and this decomposition is stable under the action of the automorphism group (which possibly permutes the isomorphic factors). In that case, the set of hyperplanes $\mathcal{W}$ is the (non-trivial) disjoint union $\mathcal{W}_0\sqcup\dots\sqcup\mathcal{W}_n$ where any $\widehat{\mathfrak{h}_i}\in\mathcal{W}_i$ meets any $\widehat{\mathfrak{h}_j}\in\mathcal{W}_j$. See \cite[Proposition 2.6]{MR2827012}. This partition of $\mathcal{W}$ induces a partition of $V\Gamma$ in the following way. If $V_i=\left\{v,\ \widehat{\mathfrak{h}}_v\in\mathcal{W}_i\right\}$ then $V\Gamma=V_0\sqcup \dots\sqcup V_n$ and for any $i,j$ distinct, $v_i\in V_i,$ $v_j\in V_j$, $\{v_i,v_j\}\in E\Gamma$. Moreover, this partition is non-trivial because $C\Gamma$ is a fundamental domain for the action of $G_\Gamma$.
\end{proof}

Let us recall that an action by automorphisms of a group $G$ on a \cat cube complex is \textbf{essential} (or $G$-\textbf{essential} if we aim to emphasize the action) if all hyperplanes are essential, that is there is no orbit at a bounded distance from an half-space.

\begin{lem}\label{lem:essential} If $\Gamma$ is not a join and has at least two vertices then the action $G_\Gamma\acts X_\Gamma$ is essential. \end{lem}

\begin{proof} For an hyperplane, to be $G_\Gamma$-essential is a $G_\Gamma$-invariant property. So it suffices to show that hyperplanes corresponding to edges with ends $\emptyset$ and $\{v\}$ (for some $v\in V\Gamma$) are essential. So let $v\in V\Gamma$. Since $\Gamma$ has at least two vertices and is not a join then there is $v^\prime\in V\Gamma$ such that $\{v,v^\prime\}$ is not an edge of $\Gamma$. Let $S=\{v,v^\prime\}$, $\widehat{\mathfrak{h}}$ be the hyperplane corresponding to edge between $\emptyset$ and $\{v\}$ and $\widehat{\mathfrak{h}^\prime}$ the one  between $\emptyset$ and $\{v^\prime\}$. These hyperplanes do not cross since $\{v,v^\prime\}$ is not an edge and none of their images under $G_S=G_v\ast G_{v^\prime}\leq G_\Gamma$. The images of the above edges under  $G_S$ span the  infinite tree without leaf $X_{\Gamma_S}$ (which is convexly embedded by Lemma~\ref{lem:conv}) and thus $\widehat{\mathfrak{h}}$ is essential.
\end{proof}

\begin{lem}\label{lem:subsets3} Let $\Gamma$ be a graph that is not a join and such that $|V\Gamma|\geq3$. Then, there is $S\subset V\Gamma$ such that $|S|=3$ and $|E\Gamma_S|\leq1$.
\end{lem}

\begin{proof} Since $\Gamma$ is not a join, $\Gamma$ is not a complete graph and there are vertices $v_1,v_2$ such that $\{v_1,v_2\}\notin E\Gamma$. Now, assume for a contradiction that for any $S\subset V\Gamma$ of cardinal 3, one has $|E\Gamma_S|\geq2$. So, for any $v_3\in V\Gamma\setminus\{v_1,v_2\}$, $\{v_1,v_3\},\{v_2,v_3\}\in E\Gamma$ and thus $\Gamma$ is the join of $\{v_1,v_2\}$ and $V\Gamma\setminus\{v_1,v_2\}$, and we have a contradiction.
\end{proof}
\begin{proof}[Proof of Theorem \ref{thm:graph_prod}] Thanks to Proposition \ref{prop:normvsquot}, a direct product is inner amenable if and only at least of its factor is. So, it suffices to prove that if $\Gamma$ is not a join and $G_\Gamma$ is inner amenable then $\Gamma$ has a unique vertex $v$ (and thus $G_\Gamma=G_v$ is inner amenable), or $\Gamma=\Gamma^1_{\textrm{euc}}$ and the vertex groups have orders 2 (that is $G_\Gamma\simeq D_\infty$).

From now on, we assume that $\Gamma$ is not a join, not reduced to a vertex nor to a pair of edges with vertex groups $\Zb_2$ and we show that in this case $G_\Gamma$ is not inner amenable. By Lemma \ref{lem:irreducible}, $X_\Gamma$ is irreducible.

Let us show that $X_\Gamma$ is not pseudo-Euclidean (in our irreducible situation, this means $\Rb$-like \cite[Lemma 7.1]{MR2827012}) and that implies that there is no invariant pair of points at infinity. That is, the action is non-elementary. So, for the sake of contradiction, let us assume, there is a $\Aut(X_\Gamma)$-invariant Euclidean subspace $E\subset X_\Gamma$. Since the fundamental domain $C\Gamma$ is compact, the projection $\pi\colon X_\Gamma\to E$ is a quasi-isometry. In particular, the subcomplexes $X\Gamma_S$ (for $S\subset V\Gamma$) can't be hyperbolic without being quasi-isometric to a real interval.

If $|V\Gamma|=2$ then at least one vertex group has order greater than 2 and then $X_\Gamma$ is tree with no leaf and at least one vertex with valency greater than 2. Thus it can't be quasi-isometric to a Euclidean space. So it remains to deal with the case where $|V\Gamma|\geq3$. Thanks to Lemma \ref{lem:subsets3}, there is $S=\{v_1,v_2,v_3\}\subset V\Gamma$ such that $|E\Gamma_S|\leq1$. If $E\Gamma_S=\emptyset$ then $X_{\Gamma_S}$ is a tree without leaf and the vertex $\emptyset$ has valency 3, which gives a contradiction. If $|E\Gamma_S|=1$, we may assume that $E\Gamma_S=\{v_1,v_2\}$. The complex $X_{\Gamma_{\{v_1,v_2\}}}$ is bounded (all squares are attached to the vertex corresponding to $G_{\{v_1,v_2\}}$) and $G_{\{v_1,v_2\}}$ and has at least 4 elements (this is similar to the last example in Table \ref{table:graphproducts}). So $X_{\Gamma_S}$ is quasi-isometric to a tree without leaf and a vertex with valency at least 4. Once again, this gives a contradiction.

So, we know that $X_\Gamma$ is not pseudo-Euclidean. By \cite[Theorem 7.2]{MR2827012}, contains a facing triple and each of this hyperplane $\widehat{\mathfrak{h}_i}$ is skewered by some contracting isometry $g_i$. If $\mathfrak{h}_1,\mathfrak{h}_2,\mathfrak{h}_3$ are the corresponding half-spaces,we may assume that $g_i\mathfrak{h}_i$ is properly contained in $\mathfrak{h}_i$.  Each contracting isometry has exactly 2 fixed points. By the configuration of the triple of hyperplanes, the three attractive points of the isometries $g_i$ are distincts and this shows that $\Gamma$ has no invariant pair of points at infinity.

By Lemmas \ref{lem:nofixedpoint} and \ref{lem:essential},  the action of $G_\Gamma$ has non-elementary and is essential. We can apply Theorem~\ref{ccc} and thus we know the existence of $X_0\subset X$ such that for any $x\in X_0$ and any conjugation invariant mean $m((G_\Gamma)_x)=1$. Let $x\in X_0$. Up to apply an element  of $G_\Gamma$, we may assume that $x$ belongs to some square containing the vertex $\emptyset$. The stabilizer of $x$ is then $G_{\sigma(x)}$ where is $\sigma(x)$ is the minimal clique appearing as vertex in the smallest cube containing $x$. We claim that there is $\gamma\in G_\Gamma$ such that $\gamma G_{\sigma(x)}\gamma^{-1}\cap G_{\sigma(x)}=\{1\}$. It follows that $m(\{1\})=1$ and thus $G_\Gamma$ cannot be inner amenable. Since $\Gamma$ is not a join, for any $v\in\sigma(x)$, there is $g_v$ in some $G_{v'}$ such that $g_v$ does commute with $G_v$. So if $\sigma(x)=\{v_1,\dots,v_n\}$, it suffices to take $\gamma=g_{v_1}\cdots g_{v_n}$.
\end{proof}

A \textbf{right-angled Artin group} is a graph product of groups where all vertex groups are infinite cyclic and a \textbf{right-angled Coxeter group} is a graph product where all vertex groups are $\Zb/2\Zb$. We readily get the following two consequences.

\begin{cor} A right-angled Artin group is inner amenable if and only if it splits as a direct product with $\Zb$.
\end{cor}

\begin{rmk} A right-angled Artin group also acts on its Davis complex which is a different \cat  cube complex from the one we use here.\end{rmk}

\begin{cor}A right-angled Coxeter group is inner amenable if and only if it splits as a direct product with the infinite dihedral group $D_\infty$.
\end{cor}

\section{Trees, amalgams and inner amenability}

In this section, we prove Theorem \ref{thm:tree}, and we also sketch a direct argument for Theorem \ref{thm:tree} which does not rely on the general results on \cat cube complexes. 

Let us say that a group action on a tree is \textbf{minimal} if there is no proper invariant subtree.

\begin{thm}\label{thm:trees} Let ${G}$ be a group acting non-elementarily and minimally on a tree $T$ and let $m$ be a conjugation invariant mean on ${G}$. Then $m({G}_x)=1$ for every $x\in T$.
\end{thm}

\begin{proof} If $T$ is reduced to a point then the result is trivial. If it is not reduced to a point, then no orbit is bounded and thus the action is essential. The conclusion therefore follows from Theorem~\ref{thm:fix_ccc}.
\end{proof}

\begin{rmk}Let us sketch a direct proof of Theorem~\ref{thm:trees} that does not rely on Theorem \ref{thm:fix_ccc}. Let $m$ be some conjugation invariant mean on ${G}$. Then we first argue that $m$ concentrates on the elliptic group elements. Otherwise, if $m(\Hyp({G}))>0$, then one can push forward the normalization of $m|_{\Hyp({G})}$ to $\partial T$ by associating to any hyperbolic element its attractive fixed point. This yields an invariant mean on $\partial T$. The removal of any one edge partitions the tree into two half spaces, which in turn yields a partition of the boundary of $T$ into two pieces. One then considers the measure of each of these two boundary pieces. There can be no edge whose associated boundary pieces each have measure 1/2, since otherwise the collection of all such edges would necessarily be a $G$-invariant segment, half-line, or line, which would contradict non-elementarity of the action. Therefore, for each edge, one of its associated boundary pieces has measure strictly greater than 1/2. By considering the intersection of all half-spaces corresponding to boundary pieces with measure strictly greater than $1/2$, we obtain a point in $T$ or in $\partial T$ which has to be fixed by ${G}$, which yields a contradiction once again. Thus we know that $m(\Ell({G}))=1$.

We claim that for any edge $e$, the measure of the point wise stabilizer of $e$ is 1. Observe that if this holds for one edge, then it in fact holds for every edge by ${G}$-invariance, convexity and minimality. So let us assume toward a contradiction that it holds for no edge. That is, for every edge $e$ of $T$, the measure of the set of elliptic elements whose fixed point set is completely contained in one of the two connected components of $T\setminus e$ is positive. By comparing the measure of the set of group elements having a fixed point set in the first connected component of $T\setminus e$ with the measure of the set of group elements having fixed point set in the second connected component of $T\setminus e$, we can argue as above (or, more concretely, as in the analysis of $\mathcal{W}_2$ in the third paragraph of the proof of Proposition \ref{ccc}) to obtain a contradiction with non-elementarity of the action.

So, we know that for any edge, the measure of its pointwise stabilizer is 1, and this implies that for any vertex its stabilizer has measure 1.  \end{rmk}

\begin{cor}\label{cor:amalgam_hnn}
Let $G = A\ast_H B$ be a nondegenerate amalgamated free product. Then every conjugation invariant mean on $G$ concentrates on $H$. Thus, $G$ is inner amenable if and only if there exist conjugation invariant, atomless means $m_A$ on $A$ and $m_B$ on $B$ with $m_A(H)=m_B(H)=1$, and $m_A(E)=m_B(E)$ for every $E\subseteq H$.

Let $G = HNN(K, H,\varphi)$ be a non-ascending HNN extension. Then every conjugation invariant mean on $G$ concentrates on $H$. Thus, $G$ is inner amenable if and only if there exists a conjugation invariant, atomless mean $m$ on $K$ with $m(H)=1$, and $m(E)=m(\phi (E))$ for every $E\subseteq H$.
\end{cor}

\begin{proof} It follows from Bass-Serre theory that in both cases (amalgams and HNN-extension), $G$ has a minimal non-elementary action on a tree such that $H$ is exactly the pointwise stabilizer of some edge. By Theorem~\ref{thm:trees}, for any conjugation invariant mean $m$ on $G$, $m(H)=1$. The respective characterizations of inner amenability are a straightforward consequence (the sufficiency of the conditions for inner amenability is obvious, and the necessity follows directly from the first part).
\end{proof}

\begin{rmk}\label{rem:Amine}
Amine Marrakchi, remarking on an earlier version of this article, informed us that he found another proof of Corollary \ref{cor:amalgam_hnn} for nondegenerate amalgams $G=A\ast _H B$, which he kindly allowed us to reproduce here.

We may assume that $|A:H|\geq 2$ and $|B:H|\geq 3$. Let $G_0 = G\setminus H$, $A_0=A\setminus H$ and $B_0 = B\setminus H$. By the definition of an amalgamated free product, the family of subsets
\[
A_0(B_0A_0)^n, \ (A_0B_0)^{n+1}, \ B_0(A_0B_0)^n, \ (B_0A_0)^{n+1}, \ \ n\geq 0
\]
forms a partition of $G_0$. Let $S=\bigcup _{n\geq 0}A_0(B_0A_0)^n \cup (A_0B_0)^{n+1}$ be the set of all elements starting with a letter in $A_0$. Take $a\in A_0$. Then we have $S\cup aSa^{-1} = G_0$. Take $b_0, b_1\in B_0$ such that $b_1^{-1}b_2 \in B_0$. Then the sets $S$, $b_1Sb_1^{-1}$ and $b_2Sb_2^{-1}$ are disjoint in $G_0$. Now suppose that $m$ is a conjugation invariant mean on $G$. Then we must have $2\mu (S) \geq m (S\cup aSa^{-1}) = m(G_0)$ and $3 m(S) = m (S \sqcup b_1Sb_1^{-1}\sqcup b_2Sb_2^{-1}) \leq m (G_0)$. This shows that $m(G_0)=0$ as we wanted.
\end{rmk}

\bibliographystyle{amsplain}
\bibliography{biblio}

\providecommand{\bysame}{\leavevmode\hbox to3em{\hrulefill}\thinspace}
\providecommand{\MR}{\relax\ifhmode\unskip\space\fi MR }
\providecommand{\MRhref}[2]{%
  \href{http://www.ams.org/mathscinet-getitem?mr=#1}{#2}
}
\providecommand{\href}[2]{#2}
\begin{thebibliography}{10}

\bibitem{BH86}
Erik B{\'e}dos and Pierre De~La~Harpe, \emph{Moyennabilit{\'e} int{\'e}rieure
  des groupes: d{\'e}finitions et exemples}, Enseign. Math.(2) \textbf{32}
  (1986), no.~1-2, 139--157.

\bibitem{BedosHarpeErratum}
Erik B\'edos and Pierre de~la Harpe, \emph{Erratum pour ``moyennabilit\'e
  int\'erieure des groupes: d\'efinitions et exemples"}, Enseign. Math. (2)
  \textbf{62} (2016), 1--2.

\bibitem{BIP18}
R{\'e}mi Boutonnet, Adrian Ioana, and Jesse Peterson, \emph{Properly proximal
  groups and their von {N}eumann algebras}, arXiv preprint arXiv:1809.01881
  (2018).

\bibitem{MR1646316}
Martin~R. Bridson, \emph{On the semisimplicity of polyhedral isometries}, Proc.
  Amer. Math. Soc. \textbf{127} (1999), no.~7, 2143--2146. \MR{1646316}

\bibitem{MR1744486}
Martin~R. Bridson and Andr\'e Haefliger, \emph{Metric spaces of non-positive
  curvature}, Grundlehren der Mathematischen Wissenschaften [Fundamental
  Principles of Mathematical Sciences], vol. 319, Springer-Verlag, Berlin,
  1999. \MR{1744486}

\bibitem{MR2558883}
Pierre-Emmanuel Caprace and Alexander Lytchak, \emph{At infinity of
  finite-dimensional {CAT}(0) spaces}, Math. Ann. \textbf{346} (2010), no.~1,
  1--21. \MR{MR2558883}

\bibitem{MR2574740}
Pierre-Emmanuel Caprace and Nicolas Monod, \emph{Isometry groups of
  non-positively curved spaces: structure theory}, J. Topol. \textbf{2} (2009),
  no.~4, 661--700. \MR{2574740}

\bibitem{MR2827012}
Pierre-Emmanuel Caprace and Michah Sageev, \emph{Rank rigidity for {CAT}(0)
  cube complexes}, Geom. Funct. Anal. \textbf{21} (2011), no.~4, 851--891.
  \MR{2827012 (2012i:20049)}

\bibitem{MR1748966}
Victor Chepoi, \emph{Graphs of some {${\rm CAT}(0)$} complexes}, Adv. in Appl.
  Math. \textbf{24} (2000), no.~2, 125--179. \MR{1748966}

\bibitem{MR1852148}
Pierre-Alain Cherix, Michael Cowling, Paul Jolissaint, Pierre Julg, and Alain
  Valette, \emph{Groups with the {H}aagerup property}, Progress in Mathematics,
  vol. 197, Birkh\"auser Verlag, Basel, 2001, Gromov's a-T-menability.
  \MR{1852148}

\bibitem{CSU16}
Ionut Chifan, Thomas Sinclair, and Bogdan Udrea, \emph{Inner amenability for
  groups and central sequences in factors}, Ergodic Theory and Dynamical
  Systems \textbf{36} (2016), no.~4, 1106--1129.

\bibitem{Ch82}
Marie Choda, \emph{Inner amenability and fullness}, Proceedings of the American
  Mathematical Society \textbf{86} (1982), no.~4, 663--666.

\bibitem{MR3589159}
F.~Dahmani, V.~Guirardel, and D.~Osin, \emph{Hyperbolically embedded subgroups
  and rotating families in groups acting on hyperbolic spaces}, Mem. Amer.
  Math. Soc. \textbf{245} (2017), no.~1156, v+152. \MR{3589159}

\bibitem{MR1709955}
Michael~W. Davis, \emph{Buildings are {${\rm CAT}(0)$}}, Geometry and
  cohomology in group theory ({D}urham, 1994), London Math. Soc. Lecture Note
  Ser., vol. 252, Cambridge Univ. Press, Cambridge, 1998, pp.~108--123.
  \MR{1709955}

\bibitem{DV18}
Tobe Deprez and Stefaan Vaes, \emph{Inner amenability, property {G}amma,
  {M}c{D}uff {II${}_{1}$} factors and stable equivalence relations}, Ergodic
  Theory and Dynamical Systems \textbf{38} (2018), no.~7, 2618--2624.

\bibitem{MR0355626}
Edward~G. Effros, \emph{Property {$\Gamma $} and inner amenability}, Proc.
  Amer. Math. Soc. \textbf{47} (1975), 483--486. \MR{0355626}

\bibitem{MR0101283}
Robert Ellis, \emph{Distal transformation groups}, Pacific J. Math. \textbf{8}
  (1958), 401--405. \MR{0101283}

\bibitem{GH91}
Thierry Giordano and Pierre de~La~Harpe, \emph{Groupes de tresses et
  moyennabilit{\'e} int{\'e}rieure}, Arkiv f{\"o}r Matematik \textbf{29}
  (1991), no.~1, 63--72.

\bibitem{HI16}
Cyril Houdayer and Yusuke Isono, \emph{Bi-exact groups, strongly ergodic
  actions and group measure space type iii factors with no central sequence},
  Communications in Mathematical Physics \textbf{348} (2016), no.~3, 991--1015.

\bibitem{IS18}
Adrian Ioana and Pieter Spaas, \emph{A class of ii $ \_1 $ factors with a
  unique mcduff decomposition}, arXiv preprint arXiv:1808.02965 (2018).

\bibitem{JS87}
Vaughan~FR Jones and Klaus Schmidt, \emph{Asymptotically invariant sequences
  and approximate finiteness}, American Journal of Mathematics (1987), 91--114.

\bibitem{KeTD19}
David Kerr and Robin Tucker-Drob, \emph{{D}ynamical alternating groups,
  stability, property {G}amma, and inner amenability}, arXiv preprint
  arXiv:1902.04131 (2019).

\bibitem{Ki12}
Yoshikata Kida, \emph{Stability in orbit equivalence for {B}aumslag-{S}olitar
  groups and {V}aes groups}, arXiv preprint arXiv:1205.5123 (2012).

\bibitem{Ki14b}
\bysame, \emph{Inner amenable groups having no stable action}, Geometriae
  Dedicata \textbf{173} (2014), no.~1, 185--192.

\bibitem{Ki14a}
\bysame, \emph{Invariants of orbit equivalence relations and
  {B}aumslag-{S}olitar groups}, Tohoku Mathematical Journal \textbf{66} (2014),
  no.~2, 205--258.

\bibitem{Ki17}
\bysame, \emph{Stable actions and central extensions}, Mathematische Annalen
  \textbf{369} (2017), no.~1-2, 705--722.

\bibitem{Ki19}
\bysame, \emph{{G}roups with infinite {FC}-radical have the {S}chmidt property
  (with an appendix by {R}obin {T}ucker-{D}rob)}, arXiv preprint
  arXiv:1901.08735 (2019).

\bibitem{KiTD18}
Yoshikata Kida and Robin Tucker-Drob, \emph{Inner amenable groupoids and
  central sequences}, arXiv preprint arXiv:1810.11569 (2018).

\bibitem{MR3619304}
Alexandre Martin, \emph{On the cubical geometry of {H}igman's group}, Duke
  Math. J. \textbf{166} (2017), no.~4, 707--738. \MR{3619304}

\bibitem{MR1389635}
John Meier, \emph{When is the graph product of hyperbolic groups hyperbolic?},
  Geom. Dedicata \textbf{61} (1996), no.~1, 29--41. \MR{1389635}

\bibitem{MR2219304}
Nicolas Monod, \emph{Superrigidity for irreducible lattices and geometric
  splitting}, J. Amer. Math. Soc. \textbf{19} (2006), no.~4, 781--814.
  \MR{2219304}

\bibitem{MvN43}
F.~J. Murray and J.~von Neumann, \emph{On rings of operators. {IV}}, Ann. of
  Math. (2) \textbf{44} (1943), 716--808. \MR{0009096}

\bibitem{Oz16}
Narutaka Ozawa, \emph{A remark on fullness of some group measure space von
  neumann algebras}, Compositio Mathematica \textbf{152} (2016), no.~12,
  2493--2502.

\bibitem{MR3329724}
Michah Sageev, \emph{{$\rm CAT(0)$} cube complexes and groups}, Geometric group
  theory, IAS/Park City Math. Ser., vol.~21, Amer. Math. Soc., Providence, RI,
  2014, pp.~7--54. \MR{3329724}

\bibitem{sisto2016contracting}
Alessandro Sisto, \emph{Contracting elements and random walks}, Journal f{\"u}r
  die reine und angewandte Mathematik (Crelles Journal) (2016).

\bibitem{MR2226017}
Yves Stalder, \emph{Moyennabilit\'e int\'erieure et extensions {HNN}}, Ann.
  Inst. Fourier (Grenoble) \textbf{56} (2006), no.~2, 309--323. \MR{2226017}

\bibitem{Tu14}
Robin Tucker-Drob, \emph{Invariant means and the structure of inner amenable
  groups}, arXiv:1407.7474.

\bibitem{Va12}
Stefaan Vaes, \emph{An inner amenable group whose von {N}eumann algebra does
  not have property {G}amma}, Acta mathematica \textbf{208} (2012), no.~2,
  389--394.

\end{thebibliography}

\end{document}